\documentclass[12pt, reqno]{amsart}
\usepackage{amsmath,amsfonts,amsbsy,amsgen,amscd,mathrsfs,amssymb,amsthm}
\usepackage[usenames,dvipsnames]{xcolor}
\usepackage[colorlinks=true,citecolor=red,linkcolor=blue]{hyperref}
\usepackage[a4paper,margin=1in]{geometry}
\usepackage{setspace}
\usepackage{amsmath}
\usepackage{bm}
\usepackage{color}
\allowdisplaybreaks[4]
\numberwithin{equation}{section}
\newtheorem{theorem}{Theorem}[section]

\newtheorem{lemma}[theorem]{Lemma}

\newtheorem{problem}[theorem]{Problem}

\title[The Minkowski problem for $k$-torsional rigidity]{The Minkowski problem with respect to the $k$-torsional rigidity associated with a $k$-Hessian 
equation}

\author{Xia Zhao and Peibiao Zhao}
\thanks{2020 Mathematics Subject Classification:  52A20 \ \ 35K96\ \ 58J35.}

\keywords{Hessian equation; $k$-torsional rigidity; curvature flow; Minkowski problem}

\begin{document}
\begin{abstract}
P. Salani [Adv. Math., 229 (2012)] introduced the $k$-torsional rigidity associated with a $k$-Hessian equation and obtained the Brunn-Minkowski inequalities $w.r.t.$ the torsional rigidity in $\mathbb{R}^3$. 

We first construct, in the present paper, a Hadamard variational formula for the $k$-torsional rigidity  with $1\leq k\leq n-1$, then we can deduce a $k$-torsional measure from the Hadamard variational formula. Based on the $k$-torsional measure, we propose the Minkowski problem for the $k$-torsional rigidity which can be equivalently converted to a nonlinear partial differential equation in smooth case:
\begin{align}\label{eq00}
f(x)=\tau|Du(\nu_\Omega^{-1}(x))|^{k+1}\sigma_{n-k}(h_{ij}+h\delta_{ij}),
\end{align}
where $\tau>0$ is a constant, $f$ is a positive smooth function defined on the unit sphere and $\sigma_{n-k}$ is the $(n-k)$-th elementary 
symmetric function of the principal curvature radii. 

Furthermore, we achieve a smooth solution to the Minkowski problem for the $k$-torsional rigidity using an appropriate Gauss curvature-type flow. A key ingredient of the argument for the existence of solutions to Gauss curvature-type flows is the uniform lower bound estimation given by the proof by contradiction in the $C^0$ estimation. The core contribution of the proof by contradiction is that it eliminates the additional conditions that the function $f$ must satisfy when studying such problems, such as symmetry by C. Chen, Y. Huang \& Y, Zhao [Math. Ann., 373(2019)], and by C. Haberl, E. Lutwak, D. Yang \& G. Y. Zhang [Adv. Math., 224(2010)] or having positive upper and lower bounds by Y. Liu, J. Lu [Trans. Amer. Math. Soc., 373(2020)].
\end{abstract}
\maketitle
\vskip 20pt
\section{Introduction and main results}

The classical Minkowski problem is a characterization problem for surface area measure $S(\Omega,\cdot)$ of a convex body $\Omega$: Given a non-zero finite Borel measure $\mu$ on unit sphere $S^{n-1}$, under what the necessary and sufficient conditions on $\mu$, does there exist an unique convex body $\Omega$ such that the given measure $\mu$ is equal to the surface area measure $S(\Omega,\cdot)$? If the given measure $\mu$ has a positive continuous density function $f$, the Minkowski problem can be equivalently regarded as the problem of prescribing the Gauss curvature in differential geometry. In general, the characterizing area measure $S_k(\Omega,\cdot)$ problem is also referred to as the Christoffel-Minkowski problem: for a given integer $1\leq k\leq n-1$ and a finite Borel measure $\mu$ on $S^{n-1}$, what are the necessary and sufficient conditions such that $\mu$ is equal to the area measure $S_{k}(\Omega,\cdot)$ of a convex body. When $k=1$, it is the Christoffel problem which was once independently solved by Firey \cite{FI1} and Berg \cite{BE}. In addition, Pogorelov \cite{PO} provided a partial result in the smooth case, Schneider \cite{SC0} get a more explicit solution with polytope case. See \cite{GO, GR} for more results. The case of $k=n-1$, the Christoffel-Minkowski problem is just the classical Minkowski problem. For $1<k<n-1$, it is a difficult and long-term open problem. Some important progress of the Christoffel-Minkowski problem was obtained by Guan and Guan \cite{GB} and Guan and Ma \cite{GM}, as well as \cite{GP, HC, ZR} and the other relevant references.

With the development of the Minkowski problem, it was extended to the $L_p$ forms by Lutwak \cite{LE0} with $p>1$, that is, the $L_p$ Minkowski problem. It is particularly crucial because the $L_p$ Minkowski problem contains some special versions. When $p=1$, it is the classical Minkowski problem; when $p=0$, it is the famous log-Minkowski problem \cite{BO}; when $p=-n$, it is the centro-affine Minkowski problem \cite{ZG}. Moreover, the solution of $L_p$ Minkowski problem plays a key role in establishing the $L_p$ affine Sobolev inequality \cite{HC2, LE01}. Interestingly, Haberl, Lutwak, Yang and Zhang \cite{HC1} proposed and studied the even Orlicz Minkowski problem in 2010 which is a more generalized form of the Minkowski problem, and its result contains the classical Minkowski problem and the $L_p$ Minkowski problem. Recently, Huang, Lutwak, Yang and Zhang \cite{HY} established the variational formulas for dual quermassintegrals with the convex hulls instead of the Wullf shapes. Therefrom, the $q$-th dual curvature measure $\widetilde{C}_q(\Omega,\cdot)$ was discovered and posed the dual Minkowski problem. Two special cases of the dual Minkowski problem include the log-Minkowski problem for $q=n$ and the Aleksandrov problem when $q=0$.

With the continuous development and enrichment of the Minkowski problems and their dual analogues, the Minkowski problem has inspired many other problems of a similar nature. Example includes the capacity Minkowski type problems which relates to the solution of boundary values problem. Early on, Jerison \cite{JE} introduced the capacity Minkowski problem related to the Dirichlet problem and through the prescribing capacity curvature measure to study this problem. Further, Xiao \cite{XJ} prescribed the capacitary curvature measures on planar convex domains. Based on these, Colesanti, Nystr\"{o}m, Salani, Xiao, Yang and Zhang \cite{CO} established the Hadamard variational formula and considered the Minkowski problem for $p$-capacity.

In addition,  another important Minkowski problem related to the solution of boundary value problem was also introduced into the Brunn-Minkowski theory, that is the Minkowski problem of the $q$-torsional rigidity. Notice that the $q$-torsional rigidity is essentially equivalent  to the existence of a solution to the $q$-Laplace equation. For convenience, we here state the definition of the $q$-torsional rigidity as follows. Let $\Omega\in\mathcal{K}^n$ (convex body), the $q$-torsional rigidity $T_q(\Omega)$ \cite{CP} with $q>1$ is defined by
\begin{align*}
\frac{1}{T_q(\Omega)}=\inf\bigg\{\frac{\int_{\Omega}|\overline{\nabla} U|^q dy}{[\int_{\Omega}|U|dy]^q}:U\in W_0^{1,q}(\Omega),\int_{\Omega}|U|dy>0\bigg\}.
\end{align*}
It is illustrated in \cite{BE0,HH} that the above functional has an unique minimizer $u\in W_0^{1,q}(\Omega)$ satisfying the following boundary value problem
\begin{align*}
\left\{
\begin{array}{lc}
\Delta_qu=-1\ \ \  \text{in} \ \ \ \Omega,\\
u=0,\ \ \ \ \ \ \ \ \text{on} \ \ \ \ \partial\Omega,\\
\end{array}
\right.
\end{align*}
where $$\Delta_qu\hat{=}{\rm div}(|\overline{\nabla} u|^{q-2}\overline{\nabla} u)$$
is the $q$-Laplace operator.

Applying the integral by part to the above $q$-Laplace equation, with the aid of Poho$\check{\textrm{z}}$aev-type identities \cite{PU}, the $q$-torsional rigidity formula can
be given by
\begin{align*}
T_q(\Omega)^{\frac{1}{q-1}}=&\frac{q-1}{q+n(q-1)}\int_{S^{n-1}}h(\Omega,\xi)d\mu^{tor}_{q}(\Omega,\xi)\\
\nonumber=&\frac{q-1}{q+n(q-1)}\int_{S^{n-1}}h(\Omega,\xi)|\overline{\nabla} u|^qdS(\Omega,\xi).
\end{align*}

When $q=2$, $T_q(\Omega)$ is the so-called torsional rigidity $T(\Omega)$ of a convex body $\Omega$ which was firstly introduced by Colesanti and Fimiani \cite{CA0}. The Minkowski problem for the torsional rigidity was studied in \cite{CA1} and extended to the $L_p$ version by Chen and Dai \cite{CZM} who proved the existence of solutions for any fixed $p>1$ and $p\neq n+2$, Hu and Liu \cite{HJ01} for $0<p<1$. Li and Zhu \cite{LN} first developed and proven the Orlicz Minkowski problem $w.r.t.$ the torsional rigidity. Huang, Song and Xu \cite{HY0} established the Hadamard variational formula for the $q$-torsional rigidity with $q>1$. Hu and Zhang \cite{HJ2} established the functional Orlicz-Brunn-Minkowski inequality for the $q$-torsional rigidity. Following the work of Hu and Zhang \cite{HJ2}, Zhao et al in \cite{ZX} have had a systematic investigation on this topic and proposed the Orlicz Minkowski problem for the $q$-torsional rigidity with $q>1$ and obtained its smooth non-even solutions by method of a Gauss curvature flow. Moreover, the authors further  in \cite{ZX1} have also posed and studied the $p$-th dual Minkowski problem for the $q$-torsional rigidity with $q>1$ and obtained the existence of smooth non-even solutions for $p<n$ $(p\neq 0)$.

Just as described in \cite {CA0}, each functional can be defined either through a variational problem, posed in a suitable space of functions, or in terms of the solution of a boundary-value problem for an elliptic operator. The first definition is in the spirit of the calculus of variation while the second reflects the point of view of elliptic PDEs. The equivalence between the two definitions relies on a well-known principle: under suitable assumptions, the minimizers of a functional are solutions of a differential equation, called the Euler-Lagrange equation of the functional itself.

Motivated by the foregoing remarkable and celebrated works, we will, in the present paper, propose and demonstrate  the Minkowski problem $w.r.t.$ the $k$-torsional rigidity associated with a $k$-Hessian equation instead of  a $q$-Laplace equation. It is believed that this research will contribute to the enrichment and development for the Minkowski problem $w.r.t.$ the torsional rigidity. Now, we recall and state firstly the concept of the $k$-torsional rigidity and its related contents below. Let $\mathcal{K}^n$ be the collection of convex bodies in Euclidean space $\mathbb{R}^n$. The set of convex bodies containing the origin in their interiors in $\mathbb{R}^n$, we write $\mathcal{K}^n_o$. Moreover,  let $C^2_{+}$ be the class of convex bodies of $C^2$ with a positive Gauss curvature at the boundary.

We consider a  $k$-Hessian equation below:
\begin{align}\label{eq101}
\left\{
\begin{array}{lc}
S_k(D^2u)=1\ \ \  \text{in} \ \ \ \Omega,\\
u=0,\ \ \ \ \ \ \ \ \text{on} \ \ \ \ \partial\Omega,\\
\end{array}
\right.
\end{align}
where $\Omega$ is a bounded convex domain of $\mathbb{R}^n$ and $S_k(D^2u)$ is the $k$-elementary symmetric function of the eigenvalues of $D^2u$, $k\in\{1,\cdots,n\}$.

Notice that, when $k=1$ in (\ref{eq101}), it is the Laplace equation, while $k=n$ in (\ref{eq101}), it is the well-known Monge-Amp\`{e}re equation. For $k\geq 2$, the $S_k$ operator is fully nonlinear and it is not elliptic unless when it is restricted to a suitable class of admissible functions, the so-called $k$-convex functions (see Section \ref{sec2} for more details).

Next, we investigate the functional $T_k$ related to problem (\ref{eq101}) which can be defined as follows (see \cite{Sa}):
\begin{align}\label{eq102}
\frac{1}{T_k(\Omega)}=\inf\bigg\{\frac{-\int_{\Omega} wS_k(D^2w) dy}{[\int_{\Omega}|w|dy]^{k+1}}:w\in \Phi _{k}^0(\Omega)\bigg\},
\end{align}
where $\Phi _{k}^0(\Omega)$ the set of admissible function that vanish on the boundary.

Note that $S_1(D^2u)=\Delta u$ and the functional $T(\Omega)$ related to $\Delta u$ is called the torsional rigidity of $\Omega$ which was defined by Colesanti \cite{CA0}, for this reason, we will refer to $T_k(\Omega)$ as the $k$-torsional rigidity of $\Omega$.

Consider the functional
\begin{align*}
J(w)=\frac{1}{k+1}\int_{\Omega}(-w)S_k(D^2w) dy-\int_{\Omega}w dy.
\end{align*}
From the works of Wang \cite{WXJ,WXJ1}, we know that $J$ has a minimizer $u\in \Phi _{k}^{0}(\Omega)$ which solves (\ref{eq101}) and also minimizers the quotient in (\ref{eq102}). Then
\begin{align}\label{eq103}
T_k(\Omega)=\frac{[-\int_{\Omega}udy]^{k+1}}{\int_{\Omega}(-u)S_k(D^2u)dy}.
\end{align}
Integrating in $\Omega$ the equation in (\ref{eq101}), we obtain
\begin{align}\label{eq104}
-\int_{\Omega}uS_k(D^2u)dy=-\int_{\Omega}udy,
\end{align}
and using (\ref{eq103}), we find the following relation
\begin{align}\label{eq105}
T_k(\Omega)=\bigg[-\int_{\Omega}u(y)dy\bigg]^k.
\end{align}

We notice that $T_k:\mathbb{R}^n\rightarrow \mathbb{R}_{+}$ is a positively homogeneous operator of degree $(n+2)k$. Indeed, if $u$ is a solution of equation (\ref{eq101}) in $\Omega$, it is easily seen that the function
\begin{align*}
v(y)=\lambda^2u(y/\lambda)
\end{align*}
solves the same problem in $\lambda\Omega$; the homogeneity of $T_k$ easily follows from (\ref{eq105}).

From (\ref{eq101}), (\ref{eq105}) and Poho$\check{\textrm{z}}$aev identity \cite[Proposition 3 in Appendix A]{BR}, the following formula is obvious,
\begin{align*}
T_k(\Omega)=\bigg(\frac{1}{k(n+2)}\int_{S^{n-1}}h(\Omega,x)|Du(\nu_\Omega^{-1}(x))|^{k+1}dS_{n-k}(\Omega,x)\bigg)^k.
\end{align*}
Denote $\widetilde{T}_k(\Omega)=(T_k(\Omega))^{\frac{1}{k}}$, then $\widetilde{T}_k:\mathbb{R}^n\rightarrow \mathbb{R}_{+}$ is a positively homogeneous operator of degree $(n+2)$, and 
\begin{align*}
\nonumber\widetilde{T}_k(\Omega)=&\frac{1}{k(n+2)}\int_{S^{n-1}}h(\Omega,x)|Du(\nu_\Omega^{-1}(x))|^{k+1}dS_{n-k}(\Omega,x),
\end{align*}
where $u$ is the solution of (\ref{eq101}) on $\Omega$, $h(\Omega,\cdot)$ is the support function of $\Omega$, $\nu_\Omega$ is the Gauss map of $\partial\Omega$ (then $\nu_\Omega^{-1}(x)$ is the point on $\partial\Omega$ where the outer normal direction is $x$) and $S_{n-k}(\Omega,\cdot)$ denotes the $(n-k)$-th area measure of $\partial\Omega$. In particularly, when $k=1$, $S_{n-1}(\Omega,\cdot)=S(\Omega,\cdot)$ is the classical surface area surface and $T_1(\Omega)$ is the torsional rigidity of $\Omega$. From the theory of convex bodies and differential geometry (see for examples \cite{SC} and \cite{UR}), we see in this case that
\begin{align*}
dS_{n-k}(\Omega,x)=\sigma_{n-k}(h_{ij}+h\delta_{ij})dx,\quad x\in S^{n-1},
\end{align*}
where $dx$ is the Lebesgue measure on $S^{n-1}$, $h$ is the support function of $\Omega$, $h_{ij}$ is the second covariant derivative of $h$ with respect to the local orthonormal frame $\{e_1,e_2,\cdots,e_{n-1}\}$ on $S^{n-1}$ and $\sigma_{n-k}(h_{ij}+h\delta_{ij})$ is the $(n-k)$-th elementary symmetric function of
the eigenvalues of $(h_{ij}+h\delta_{ij})$ and $\delta_{ij}$ is the Kronecker delta. Thus
\begin{align}\label{eq106}
\widetilde{T}_k(\Omega)=\frac{1}{k(n+2)}\int_{S^{n-1}}h(\Omega,x)|Du(\nu_\Omega^{-1}(x))|^{k+1}\sigma_{n-k}(h_{ij}+h\delta_{ij})dx.
\end{align}

Motivated by foregoing successful studies on the torsional rigidity and the torsional measures, we now define the $k$-torsional measure $\mu^{tor}_k(\Omega,\cdot)$ of $\Omega\in\mathcal{K}_o^n$ as follows. A variational formula for the $k$-torsional rigidity will be established in Section \ref{sec3} which can also induce the $k$-torsional measure. In fact, the $k$-torsional measure is the differential of the $k$-torsional rigidity.

Define
\begin{align}\label{eq107}
\mu^{tor}_k(\Omega,\eta)=\int_{\eta}|Du(\nu_\Omega^{-1}(x))|^{k+1}dS_{n-k}(\Omega,x),
\end{align}
then
\begin{align}\label{eq108}
d\mu^{tor}_k(\Omega,\cdot)=|Du(\nu_\Omega^{-1}(x))|^{k+1}\sigma_{n-k}(h_{ij}+h\delta_{ij})dx.
\end{align}

Based on the $k$-torsional measure, we can naturally propose the normalised Minkowski problem for the $k$-torsional rigidity.
\begin{problem}\label{pro11}
Let $1\leq k\leq n-1$. Given a non-zero finite Borel measure $\mu$ on $S^{n-1}$, under what the necessary and sufficient conditions on $\mu$, does there exist an unique convex body $\Omega$ and a positive constant $\tau$ such that $\tau\mu_k^{tor}(\Omega,\cdot)=\mu$?
\end{problem}

If the given measure $\mu$ is absolutely continuous with respect to the Lebesgue measure and $\mu$ has a smooth density function $f:S^{n-1}\rightarrow (0,\infty)$. From (\ref{eq108}), then solving Problem \ref{pro11} can be equivalently viewed as solving the following nonlinear partial differential equation on $S^{n-1}$:
\begin{align*}
f(x)=\frac{d\mu}{dx}=\frac{\tau|Du(\nu_\Omega^{-1}(x))|^{k+1}\sigma_{n-k}(h_{ij}+h\delta_{ij})dx}{dx}.
\end{align*}
In the present paper, we solve Problem \ref{pro11} by studying the following nonlinear equation:
\begin{align}\label{eq109}
f(x)=\tau|Du(\nu_\Omega^{-1}(x))|^{k+1}\sigma_{n-k}(h_{ij}+h\delta_{ij}).
\end{align}

If $k=1$, Problem \ref{pro11} is the Minkowski problem of the torsional rigidity, correspondingly, (\ref{eq109}) is a partial differential equation of the Minkowski problem for the torsional rigidity.

In the present paper, we will investigate the existence of smooth solutions to the normalized Minkowski problem for the $k$-torsional rigidity by the method of curvature flows which contains the Gauss curvature flow and mean curvature flow etc., where the Gauss curvature flow was first introduced and studied by Firey \cite{FI} to model the shape change of worn stones. Since then, the Gauss curvature flow has been widely used to find smooth solutions of  Minkowski problems. In general, the existence of smooth solutions to the corresponding Minkowski problems can be obtained by curvature flows under appropriate assumption conditions on density function $f$. For example, Liu and Lu \cite{LY} solved the dual Orlicz-Minkowski problem and obtained the existence of smooth solutions under the condition of the restricted density function by a Gauss curvature flow. Moreover, Chen, Huang and Zhao \cite{CC} attained the existence of smooth solutions to the $L_p$ dual Minkowski problem under the condition of even $f$ by a Gauss curvature flow. In addition,  the mean curvature flow play a crucial role in the study of geometric inequalities. Wang, Weng and Xia \cite{WGF} constructed a new locally constrained curvature flow to obtain Alexandrov-Fenchel inequalities for convex hypersurfaces in the half-space with capillary boundary. Wei and Xiong \cite{WY} gave new proof of a class of Alexandrov-Fenchel inequalities for anisotropic mixed volumes of smooth convex domains in Euclidean space by a volume-preserving anisotropic mean curvature type flow. Cui and Zhao \cite{CJS} constructed and confirmed  the sharp Micheal-Simon inequalities by a class of mean curvature type flows. The research results on curvature flows and their applications are extremely fruitful. In particular, the mean curvature flow, Ricci curvature flow, and so on. One can refer to \cite{AB, BA, CH, CK0, HJ, HU, LR, MR,ZX0} and the references therein for details.

In this article, we confirm the existence of smooth solutions to the Minkowski problem for the $k$-torsional rigidity by method of a curvature flow. Let $\partial\Omega_0$ be a smooth, closed and strictly convex hypersurface in $\mathbb{R}^n$ containing the origin in its interior and $f$ is a positive smooth function on $S^{n-1}$. We construct and consider the long-time existence and convergence of a following curvature flow which is a family of convex hypersurfaces $\partial\Omega_t$ parameterized by smooth maps $X(\cdot ,t):
S^{n-1}\times (0, \infty)\rightarrow \mathbb{R}^n$
satisfying the initial value problem:
\begin{align}\label{eq110}
\left\{
\begin{array}{lc}
\frac{\partial X(x,t)}{\partial t}=\frac{\langle X,v\rangle}{f(v)}|D u(X,t)|^{k+1}\sigma_{n-k}
(x,t)v-\eta(t)X(x,t),  \\
X(x,0)=X_0(x),\\
\end{array}
\right.
\end{align}
where $\sigma_{n-k}(x,t)$ is the $(n-k)$-th ($1\leq k\leq n-1$) elementary symmetric function for principal curvature radii, $v=x$ is the
outer unit normal at $X(x,t)$, $\langle X,v \rangle$ represents the standard inner product of $X$ and $v$, $\eta(t)$ is given by
\begin{align}\label{eq111}
\eta(t)=\frac{\int_{S^{n-1}}h(x,t)\sigma_{n-k}(x,t)|D u(X,t)|^{k+1}dx}{\int_{S^{n-1}}h(x,t)f(x)dx}.
\end{align}

For convenience, we introduce the following functional which is very important for $C^0$ estimate of solution to the curvature flow (\ref{eq110}).
\begin{align}\label{eq112}
\Phi(\Omega_t)=\int_{S^{n-1}}h(x,t)f(x)dx.
\end{align}

We obtain the long-time existence and convergence of the flow (\ref{eq110}) in this article, see Theorem \ref{thm12} for details.

\begin{theorem}\label{thm12}
Let $1\leq k\leq n-1$, $u(x,t)$ be a smooth admissible solution of (\ref{eq101}) in $\Omega_t$ and $\partial\Omega_0$ be a smooth, closed and strictly convex hypersurface in $\mathbb{R}^n$ containing the origin in its interior, $f$ is a positive smooth function on $S^{n-1}$. Then the flow (\ref{eq110}) has an unique smooth convex solution $\partial\Omega_t=X(S^{n-1},t)$. Moreover, when $t\rightarrow\infty$, there is a subsequence of $\partial\Omega_t$ converges in $C^{\infty}$ to a smooth, closed and strictly convex hypersurface $\partial\Omega_\infty$, the support function of convex body $\Omega_\infty$ enclosed by $\partial\Omega_\infty$ satisfies (\ref{eq109}).
\end{theorem}

This paper is organized as follows. We collect some necessary background materials about convex bodies in Section \ref{sec2}. In Section \ref{sec3}, we establish a Hadamard variational formula for the $k$-torsional rigidity and discuss properties of two geometric functionals along the flow (\ref{eq110}). In Section \ref{sec4}, we give the priori estimates for solutions to the flow (\ref{eq110}). We obtain the long-time existence and convergence of the flow (\ref{eq110}) and complete the proof of Theorem \ref{thm12} in Section \ref{sec5}.

\section{\bf Preliminaries}\label{sec2}
In this subsection, we give a brief review of some relevant notions and terminologies.

\subsection{Convex hypersurface} (see \cite{SC,UR}) Let $\mathbb{R}^n$ be the $n$-dimensional Euclidean space, and let $S^{n-1}$ be the unit sphere in $\mathbb{R}^n$. The origin-centered unit ball $\{y\in\mathbb{R}^n:|y|\leq 1\}$ is always denoted by $B$. Write $\omega_n$ for the volume of $B$ and recall that its surface area is $n\omega_n$.

Let $\partial\Omega$ be a smooth, closed and strictly convex hypersurface containing the origin in its interior. The support function of a convex body $\Omega$ enclosed by $\partial\Omega$ is defined by
\begin{align*}h_\Omega(x)=h(\Omega,x)=\max\{x\cdot y:y\in\Omega\},\quad \forall x\in S^{n-1},\end{align*}
and the radial function of $\Omega$ with respect to $o$ (origin) $\in\mathbb{R}$ is defined by
\begin{align*}\rho_{\Omega}(v)=\rho(\Omega,v)=\max\{c>0:cv\in\Omega\},\quad  v\in S^{n-1}.\end{align*}
We easily obtain that the support function is homogeneous of degree $1$ and the radial function is homogeneous of degree $-1$.

For a convex body $\Omega\in\mathbb{R}^n$, its support hyperplane with an outward unit normal vector $\xi\in S^{n-1}$ is represented by
\begin{align*}
H(\Omega,\xi)=\{y\in\mathbb{R}^n:y\cdot \xi=h(\Omega,\xi)\}.
\end{align*}
A boundary point of $\Omega$ which only has one supporting hyperplane is called a regular point, otherwise, it is a singular point. The set of singular points is denoted as $\sigma \Omega$, it is
well known that $\sigma \Omega$ is a set with a spherical Lebesgue measure 0.

For a Borel set $\eta\subset S^{n-1}$, its surface area measure is defined as
\begin{align*}S(\Omega,\eta)=\mathcal{H}^{n-1}(\nu^{-1}_{\Omega}(\eta)),\end{align*}
where $\mathcal{H}^{n-1}$ is the $(n-1)$-dimensional Hausdorff measure. The Gauss map $\nu_\Omega:y\in\partial \Omega\setminus \sigma \Omega\rightarrow S^{n-1}$ is represented by
\begin{align*}\nu_\Omega(y)=\{\xi\in S^{n-1}:y\cdot \xi=h_\Omega(\xi)\}.\end{align*}
Here, $\partial \Omega\setminus \sigma \Omega$ is abbreviated as $\partial^\prime\Omega$, something we will often do.

Correspondingly, for a Borel set $\eta\subset S^{n-1}$, its inverse Gauss map is denoted by $\nu_\Omega^{-1}$,
\begin{align*}\nu_\Omega^{-1}(\eta)=\{y\in\partial \Omega:\nu_\Omega(y)\in\eta\}.\end{align*}
For $g\in C(S^{n-1})$, there holds
\begin{align*}
\int_{\partial\Omega\setminus \sigma\Omega}g(\nu_\Omega(y))d\mathcal{H}^{n-1}(y)=\int_{S^{n-1}}g(v)dS(\Omega,v).
\end{align*}

Suppose that $\Omega$ is parameterized by the inverse Gauss map $X:S^{n-1}\rightarrow \Omega$, that is $X(x)=\nu_\Omega^{-1}(x)$. Then the support function $h$ of $\Omega$ can be computed by
\begin{align}\label{eq201}h(x)=x\cdot X(x) , \ \ x\in S^{n-1},\end{align}
where $x$ is the outer normal of $\Omega$ at $X(x)$. Let $\{e_1,\cdots,e_{n-1}\}$ be an orthogonal frame on $S^{n-1}$. Let $\nabla$ be the gradient operator on $S^{n-1}$. Differentiating (\ref{eq201}), we arrive at
\begin{align*}
\nabla_ih=\langle \nabla_ix,X(x)\rangle + \langle x,\nabla_iX(x)\rangle.
\end{align*}
Since $\nabla_iX(x)$ is tangent to $\partial\Omega$ at $X(x)$, we have
\begin{align*}
\nabla_ih=\langle \nabla_ix,X(x)\rangle.
\end{align*}
It follows that
\begin{align}\label{eq202}
\overline{\nabla} h=\nabla h+hx=X(x).
\end{align}
$\overline{\nabla} h$ is regarded essentially as the point on $\partial\Omega$ whose outer unit normal vector is $x\in S^{n-1}$.

Denote by $h_i$ and $h_{ij}$ the first and second order covariant derivatives of $h$
on $S^{n-1}$, then computing as in \cite{JE0}, one can get
\begin{align}\label{eq203}
X(x)=h(x)_ie_i+h(x)x,\ \ \ \ X_i(x)=\omega_{ij}e_j,
\end{align}
where $\omega_{ij}=h_{ij}+h\delta_{ij}$. Note that we use the summation convention for the repeated indices here and after.

Since $\Omega$ is of class $C_+^2$, the matrix $\omega_{ij}$ is symmetric and positive definite. For convenience, we denote
\begin{align*}
\mathbf{H}=\{h\in C^2(S^{n-1}):(h_{ij}+h\delta_{ij})~~\text{is positive definite}\}.
\end{align*}
Indeed, $\mathbf{H}$ is formed exactly by support functions of convex bodies of class $C_+^2$.

In the following, we let $\Omega$ and $\Omega^\prime$ be two convex domains of $C^2_+$, and let $h$ and $\theta$ be support functions of $\Omega$ and $\Omega^\prime$, respectively. Denote $h_t=h+t\theta$ by the support function of $\Omega_t\hat{=}\Omega+t\Omega^\prime$, then for $|t|$ small enough, $h+t\theta\in\mathbf{H}$ as well. Hence, $\Omega_t$ is a convex domain of class $C^2_+$ with support function $h_t=h+t\theta$.

\subsection{Symmetric functions and Hessian operators} For $k\in \{1,\cdots,n\}$, the $k$-th elementary symmetric function of $A$ is defined as
\begin{align*}
S_k(A)=S(\lambda_1,\cdots,\lambda_n)=\sum_{1\leq i_1<\cdots<i_k\leq n}\lambda_{i_1}\cdots\lambda_{i_k},
\end{align*}
where $A=(a_{ij})$ is a matrix in the space $\mathcal{S}_n$ of the real symmetric $n\times n$ matrices and $\lambda_1,\cdots,\lambda_n$ are the eigenvalues of $A$. Notice that $S_k(A)$ is just the sum of all $k\times k$ principal minors of $A$. In particularly, $S_1(A)=\rm tr A$ is the trace of $A$ and $S_n(A)=\det(A)$ is its determinant.

The operator $S_k^{\frac{1}{k}}$, for $k\in \{1,\cdots,n\}$ is homogeneous of degree $1$ and it is increasing and concave if it is restricted to
\begin{align*}
\Gamma_k=\{A\in\mathcal{S}_n: S_i(A)> 0~~\text{for}~~ i=1,\cdots,k\}.
\end{align*}

Denoting by
\begin{align*}
S_k^{ij}(A)=\frac{\partial}{\partial a_{ij}}S_k(A),
\end{align*}
Euler identity for homogeneous functions shows that there holds
\begin{align*}
S_k(A)=\frac{1}{k}S_k^{ij}(A)a_{ij},
\end{align*}
we adopt  the Einstein summation convention for repeated indices here and full text.

Let $\Omega$ be an open subset of $\mathbb{R}^n$ and let $u\in C^2(\Omega)$. The $k$-Hessian operator $S_k(D^2u)$ is defined as the $k$-th elementary symmetric function of $D^2u$. Note that
\begin{align*}
S_1(D^2u)=\Delta u \quad \text{and} \quad S_n(D^2u)=\det(D^2 u).
\end{align*}
For $k>1$, the $k$-Hessian operators are fully nonlinear and, in general, not elliptic, unless it is restricted to the class of $k$-convex functions:
\begin{align*}
\Phi_k^2(\Omega)=\{u\in C^2(\Omega):S_i(D^2u)\geq 0~\text{in}~\Omega, i=1,2,\cdots,k\}.
\end{align*}
Notice that $\Phi^2_n(\Omega)$ coincides with class of $C^2(\Omega)$ convex functions.

A direct computation yields that $(S_k^{1j}(D^2u),\cdots,S_k^{nj}(D^2(u))$ is divergence free (see \cite{RE}), namely,
\begin{align*}
\frac{\partial}{\partial x_i}S_k^{ij}=0,
\end{align*}
hence, $S_k(D^2u)$ can be written in divergence form
\begin{align*}
S_k(D^2u)=\frac{1}{k}S_k^{ij}(D^2u)u_{ij}=\frac{1}{k}(S_k^{ij}(D^2u)u_j)_i,
\end{align*}
where subscripts $i,j$ stand for partial differentiations. For example, when $k=1$, we have $S_1^{ij}=\delta_{ij}$ and $S_1(D^2u)=\delta_{ij}u_{ij}$.

Let $\Omega$ be a bounded connected domain of $\mathbb{R}^n$ of class $C^2$ having principal curvatures $\kappa=(\kappa_1,\cdots,\kappa_{n-1})$ and outer unit normal vector $v_x$. For $k=1,\cdots,n-1$, we define the $k$-th curvature of $\partial \Omega$ by
\begin{align*}
\sigma_k(\partial \Omega)=\sigma_k(\kappa_1,\cdots,\kappa_{n-1}).
\end{align*}
Moreover, we set
\begin{align*}
\sigma_0=S_0\equiv 1, \quad \sigma_n\equiv 0.
\end{align*}
For example, $\sigma_1$ is equal to $(n-1)$-time the mean curvature of $\partial \Omega$, while $\sigma_{n-1}$ is the Gauss curvature of $\partial \Omega$.

In analogy with the case of functions, $\Omega$ is said $k$-convex, with $ k\in \{1,\cdots,{n-1}\}$, if $\sigma_j\geq 0$ for $j = 1,\cdots,k$ at every point $y=\partial \Omega$. We recall here that any sublevel set of a $k$-convex function is $(k-1)$-convex (see \cite{CA}).

In general, for $1\leq k\leq n$, a straightforward calculation yields
\begin{align*}
S_k(D^2u)=\sigma_k|Du|^k+\frac{S_ku_iu_lu_{lj}}{|Du|^2}.
\end{align*}
In addition, the following pointwise identity holds (see \cite{RE})
\begin{align*}
\sigma_{k-1}=\frac{S_k^{ij}(D^2u)u_iu_j}{|Du|^{k+1}}.
\end{align*}

\section{\bf Variational formula and associated functionals of the geometric flow}\label{sec3}
In this subsection, we give the scalar form of the geometric flow (\ref{eq110}) and provide a Hadamard variational formula of the $k$-torsional rigidity $\widetilde{T}_k(\Omega)$, as well as discuss geometric characteristic of associated functionals along the flow to solve the  Minkowski problem for the $k$-torsional rigidity with $1\leq k\leq n-1$.

Taking the scalar product of both sides of the equation and of the initial condition in the flow (\ref{eq110}) by $v$, by means of the definition of support function (\ref{eq201}), we describe the flow (\ref{eq110}) with the support function as the following quantity equation,
\begin{align}\label{eq301}
\left\{
\begin{array}{lc}
\frac{\partial h(x,t)}{\partial t}=\frac{h}{f(x)}|D u(X(x,t),t)|^{k+1}\sigma_{n-k}
(x,t)-\eta(t)h(x,t), \\
h(x,0)=h_0(x).\\
\end{array}
\right.
\end{align}

Firstly, we provide a Hadamard variational formula of the $k$-torsional rigidity $\widetilde{T}_k(\Omega)$.
\begin{lemma}\label{lem31}
Let $\Omega$ and $\Omega^\prime$ be two convex domains of $C^2_+$, and $h$ and $\theta$ be support functions of $\Omega$ and $\Omega^\prime$, respectively. Let $\Omega_t\hat{=}\Omega+t\Omega^\prime$ with support function $h_t=h+t\theta$. Suppose $u(X,t)$ is the solution to (\ref{eq101}) in $\Omega_t$. Then we have
\begin{align}\label{eq302}
\frac{d}{dt}\widetilde{T}_k(\Omega_t)\bigg|_{t=0}=&\frac{1}{k}\int_{S^{n-1}}\theta(x)|Du(X(x))|^{k+1}\sigma_{n-k}(h_{ij}(x)+h(x)\delta_{ij})dx\\
\nonumber=&\frac{1}{k}\int_{S^{n-1}}\theta(x)d\mu_k^{tor}(\Omega,x).
\end{align}
Here, $\mu_k^{tor}(\Omega,\cdot)$ is the $k$-torsional measure of $\Omega$, see (\ref{eq107}).
\end{lemma}
\begin{proof}
The idea of proof is referenced in  \cite[Proposition 2.5]{HY0}. For convenience, we denote $W=\omega_{ij}(x)=h_{ij}(x)+h(x)\delta_{ij}$ and
\begin{align}\label{eq303}
\mathcal{G}(h_t)=|Du(X(x,t),t)|^{k+1}\sigma_{n-k}((h_t)_{ij}(x)+h_t(x)\delta_{ij}), \ \ \ \ x\in S^{n-1}.
\end{align}
Then we have the following result by (\ref{eq106}),
\begin{align}\label{eq304}
\frac{d}{dt}\widetilde{T}_k(\Omega_t)\bigg|_{t=0}=\frac{1}{k(n+2)}\int_{S^{n-1}}\bigg(\theta(x)\mathcal{G}(h)(x)+h(x)\frac{d}{dt}\mathcal{G}(h_t)\bigg|_{t=0}(x)\bigg)dx.
\end{align}
Now, we try to compute the second term on the right hand side of (\ref{eq304}). Note that
\begin{align}\label{eq304+}
d_{ij}=\frac{\partial\sigma_{n-k}(W(x))}{\partial\omega_{ij}}=\sum_{p=1}^{n}\frac{\partial\sigma_{n-k}(W(x))}{\partial\lambda_p}\frac{\partial\lambda_p}{\partial \omega_{ij}}=\sum_{p=1}^{n}\sigma_{k-1}^{(p)}(W)v_p^iv_p^j,
\end{align}
where $\sigma_{k-1}^{(p)}(\omega_{ij})$ is the $(k-1)$-order elementary symmetric polynomial after removing the $p$-th eigenvalue and $v_p=(v_p^1,v_p^2,\cdots,v_p^n)^T$ is unit eigenvector of $W$ (The corresponding eigenvalue is $\lambda_p$).
\begin{align}\label{eq305}
\frac{d}{dt}\mathcal{G}(h_t)\bigg|_{t=0}(x)=&|Du(X(x))|^{k+1}d_{ij}(x)(\theta_{ij}(x)+\theta(x)\delta_{ij})\\
\nonumber&+(k+1)|Du(X(x))|^k\frac{d}{dt}|Du(X(x,t),t)|\bigg|_{t=0}\sigma_{n-k}(\omega_{ij}(x)).
\end{align}
Since $u(X,t)=0$ on $\partial\Omega_t$ and $u(X,t)>0$ in $\Omega_t$, it gives that
\begin{align}\label{eq306}
|Du(X(x,t),t)|=-\langle Du(X(x,t),t),x\rangle.
\end{align}
Then
\begin{align*}
\frac{d}{dt}&|Du(X(x,t),t)|\bigg|_{t=0}=-\langle D^2u(X(x))\dot{X}(x),x\rangle-\langle D\dot{u}(X(x))x\rangle\\
=&-\langle D^2u(X(x))[\theta_i(x)e_i+\theta(x)x],x\rangle-\langle D\dot{u}(X(x))x\rangle\\
=&\frac{d_{ij}(x)}{\sigma_{n-k}(\omega_{ij}(x))}|Du(X(x))|_j\theta_i(x)-\langle D^2u(X(x))x,x)\rangle \theta(x)-\langle D\dot{u}(X(x))x\rangle,
\end{align*}
where we use (ii) of Lemma \ref{lem47}, and $\dot{X}(x)=\theta_i(x)e_i+\theta(x)x$ from (\ref{eq203}). Putting the above equality into (\ref{eq305}), we obtain
\begin{align*}
\frac{d}{dt}&\mathcal{G}(h_t)\bigg|_{t=0}(x)=[|Du(X(x))|^{k+1}d_{ij}(x)\theta_i(x)]_j\\
&+\bigg(|Du(X(x))|^{k+1}d_{ii}(x)-(k+1)|Du(X(x))|^k\sigma_{n-k}(\omega_{ij}(x))\langle D^2u(X(x))x,x\rangle\bigg)\theta(x)\\
&-(k+1)|Du(X(x))|^k\sigma_{n-k}(\omega_{ij}(x))\langle D\dot{u}(X(x)),x\rangle.
\end{align*}
Here, we use $d_{ij,j}=0$, see \cite{TR}.

Denote $\frac{d}{dt}\mathcal{G}(h_t)\bigg|_{t=0}(x)=\mathcal{L}\theta=\mathcal{L}_1\theta+\mathcal{L}_2\theta+\mathcal{L}_3\theta$ with
\begin{align*}
\mathcal{L}_1\theta &=[|Du(X(x))|^{k+1}d_{ij}(x)\theta_i(x)]_j,\\
\mathcal{L}_2\theta &=\bigg(|Du(X(x))|^{k+1}d_{ii}(x)-(k+1)|Du(X(x))|^k\sigma_{n-k}(\omega_{ij}(x))\langle D^2u(X(x))x,x\rangle\bigg)\theta(x),\\
\mathcal{L}_3\theta &=-(k+1)|Du(X(x))|^k\sigma_{n-k}(\omega_{ij}(x))\langle D\dot{u}(X(x)),x\rangle.
\end{align*}

We can see that $\mathcal{L}$ is self-adjoint on $L^2(S^{n-1})$, i.e.
\begin{align}\label{eq307}
\int_{S^{n-1}}\theta^1\mathcal{L}\theta^2=\int_{S^{n-1}}\theta^2\mathcal{L}\theta^1.
\end{align}
In fact, $\mathcal{L}_2$ is self-adjoint obviously. $\mathcal{L}_1$ can be easily seen to be self-adjoint with an integration by parts. And $\mathcal{L}_3$ is self-adjoint by \cite[Lemma 2.3]{HY0}.

By the $(n+1)$-homogeneity of $\mathcal{G}(h)$, it yields that
\begin{align}\label{eq308}
\mathcal{L}h=(n+1)\mathcal{G}(h).
\end{align}
Therefore, using (\ref{eq307}) and (\ref{eq308}), (\ref{eq304}) gives that
\begin{align*}
\frac{d}{dt}\widetilde{T}_k(\Omega_t)\bigg|_{t=0}(x)=&\frac{1}{k(n+2)}\int_{S^{n-1}}\bigg(\theta(x)\mathcal{G}(h)(x)+h(x)\mathcal{L}\theta(x)\bigg)dx\\
=&\frac{1}{k(n+2)}\int_{S^{n-1}}\bigg(\theta(x)\mathcal{G}(h)(x)+\theta(x)\mathcal{L}h(x)\bigg)dx\\
=&\frac{1}{k(n+2)}\int_{S^{n-1}}\bigg(\theta(x)\mathcal{G}(h)(x)+\theta(x)(n+1)\mathcal{G}(h)\bigg)dx\\
=&\frac{1}{k}\int_{S^{n-1}}\theta(x)\mathcal{G}(h)(x)dx\\
=&\frac{1}{k}\int_{S^{n-1}}\theta(x)d\mu_k^{tor}(\Omega,x).
\end{align*}
\end{proof}

Next, we discuss the  geometric characteristic of functionals $\widetilde{T}_k(\Omega_t)$ and $\Phi(\Omega_t)$ with respect to the flow (\ref{eq110}). We will first show that the functional $\widetilde{T}_k(\Omega_t)$ is non-decreasing along the flow (\ref{eq110}).
\begin{lemma}\label{lem32}
The functional $\widetilde{T}_k(\Omega_t)$ is non-decreasing along the flow (\ref{eq110}). Namely, $\frac{\partial}{\partial t}\widetilde{T}_k(\Omega_t)\geq0$, the equality holds if and only if the support function of $\Omega_t$ satisfies (\ref{eq109}).
\end{lemma}
\begin{proof}
By Lemma \ref{lem31}, we know that
\begin{align*}
\frac{d}{dt}\widetilde{T}_k(\Omega_t)=\frac{1}{k}\int_{S^{n-1}}\frac{\partial h(x,t)}{\partial t}|Du(X(x,t),t)|^{k+1}\sigma_{n-k}(x,t)dx.
\end{align*}
Thus, from (\ref{eq301}), (\ref{eq111}) and H\"{o}lder inequality, we obtain the following result,
\begin{align*}
&\frac{d}{dt}\widetilde{T}_k(\Omega_t)\\
=&\frac{1}{k}\bigg(\int_{S^{n-1}}\frac{h}{f(x)}|Du|^{2(k+1)}(\sigma_{n-k})^2dx\\
&-\frac{\int_{S^{n-1}}h\sigma_{n-k}|D u|^{k+1}dx}{\int_{S^{n-1}}hf(x)dx}\int_{S^{n-1}}h|Du|^{k+1}\sigma_{n-k}dx\bigg)\\
=&\frac{1}{k}\bigg\{\frac{1}{\int_{S^{n-1}}hf(x)dx}\bigg[\int_{S^{n-1}}hf(x)dx\int_{S^{n-1}}\frac{h}{f(x)}|Du|^{2(k+1)}(\sigma_{n-k})^2dx\\
&-\bigg(\int_{S^{n-1}}h|Du|^{k+1}\sigma_{n-k}dx\bigg)^2\bigg]\bigg\}\\
=&\frac{1}{k}\bigg[\frac{1}{\int_{S^{n-1}}hf(x)dx}\bigg\{\bigg[\bigg(\int_{S^{n-1}}\bigg(\bigg(hf(x)\bigg)^{\frac{1}{2}}\bigg)^2dx\bigg)^{\frac{1}{2}}\\
&\bigg(\int_{S^{n-1}}\bigg(\bigg(\frac{h}{f(x)}|Du|^{2(k+1)}(\sigma_{n-k})^2\bigg)^{\frac{1}{2}}\bigg)^2dx\bigg)^{\frac{1}{2}}\bigg]^2
-\bigg(\int_{S^{n-1}}h\sigma_{n-k}|Du|^{k+1}dx\bigg)^2\bigg\}\bigg]\\
\geq&\frac{1}{k}\bigg(\frac{1}{\int_{S^{n-1}}hf(x)dx}\bigg\{\bigg[\int_{S^{n-1}}\bigg(hf(x)\bigg)^{\frac{1}{2}}\times
\bigg(\frac{h|Du|^{2(k+1)}(\sigma_{n-k})^2}{f(x)}\bigg)^{\frac{1}{2}}dx\bigg]^2\\
&-\bigg(\int_{S^{n-1}}h\sigma_{n-k}|Du|^{k+1}dx\bigg)^2\bigg\}\bigg)\\
=&0.
\end{align*}
By the equality condition of H\"{o}lder inequality, we know that the above equality holds if and only if $\bigg(hf(x)\bigg)^{\frac{1}{2}}=\tau\bigg(\frac{h|Du|^{2(k+1)}(\sigma_{n-k})^2}{f(x)}\bigg)^{\frac{1}{2}}$, i.e.
\begin{align*}
\tau\sigma_{n-k}|Du|^{k+1}=f(x).
\end{align*}
Namely, the support function of $\Omega_t$ satisfies (\ref{eq109}) with $\tau=\frac{1}{\eta(t)}$.
\end{proof}

Moreover, we prove the functional (\ref{eq112}) is unchanged along the flow (\ref{eq110}). Please refer to the following lemma for details.
\begin{lemma}\label{lem33}The functional (\ref{eq112}) is unchanged along the flow (\ref{eq110}). That is,  $\frac{d}{d t}\Phi(\Omega_t)=0$.
\end{lemma}
\begin{proof} According to (\ref{eq112}), (\ref{eq111}) and (\ref{eq301}), we obtain
\begin{align*}
&\frac{\partial}{\partial t}\Phi(\Omega_t)\\
=&\int_{S^{n-1}}f(x)\frac{\partial h}{\partial t}dx\\
=&\int_{S^{n-1}}f(x)\bigg(\frac{h}{f(x)}|Du|^{k+1}\sigma_{n-k}-\eta(t)h\bigg)dx\\
=&\int_{S^{n-1}}h\sigma_{n-k}|Du|^{k+1}dx-\frac{\int_{S^{n-1}}h\sigma_{n-k}|Du|^{k+1}dx}{\int_{S^{n-1}}hf(x)dx}\int_{S^{n-1}}hf(x)dx\\
=&0.
\end{align*}
This ends the proof of Lemma \ref{lem33}.
\end{proof}

\section{\bf A priori estimates}\label{sec4}

In this subsection, we establish the $C^0$, $C^1$ and $C^2$ estimates for solutions to the equation (\ref{eq301}). In the following of this paper, we always assume that $\partial\Omega_0$ is a smooth, closed and strictly convex hypersurface in $\mathbb{R}^n$ containing the origin in its interior. $h:S^{n-1}\times [0,T)\rightarrow \mathbb{R}$ is a smooth solution to the equation (\ref{eq301}) with the initial $h(\cdot,0)$ the support function of $\partial\Omega_0$. Here, $T$ is the maximal time for existence of smooth solutions to the equation (\ref{eq301}).

\subsection{$C^0$, $C^1$ estimates}

In order to complete the $C^0$ estimate, we firstly need to introduce the following lemma for convex bodies.
\begin{lemma}\label{lem41}\cite[Lemma 2.6]{CH}
Let $\Omega\in\mathcal{K}^n_o$, $h$ and $\rho$ are respectively the support function and the radial function of $\Omega$, and $x_{\max}$ and $\xi_{\min}$ are two points such that
$h(x_{\max})=\max_{S^{n-1}}h$ and $\rho(\xi_{\min})=\min_{S^{n-1}}\rho$. Then
\begin{align*}
\max_{S^{n-1}}h=&\max_{S^{n-1}}\rho \quad \text{and} \quad \min_{S^{n-1}}h=\min_{S^{n-1}}\rho;\end{align*}
\begin{align*}h(x)\geq& x\cdot x_{\max}h(x_{\max}),\quad \forall x\in S^{n-1};\end{align*}
\begin{align*}\rho(\xi)\xi\cdot\xi_{\min}\leq&\rho(\xi_{\min}),\quad \forall \xi\in S^{n-1}.\end{align*}
\end{lemma}
\begin{lemma}\label{lem42}
Let $\partial\Omega_t$ be a smooth strictly convex solution to the flow (\ref{eq110}) in $\mathbb{R}^n$ and $u(X,t)$ be the smooth admissible solution of (\ref{eq101}) in $\Omega_t$, $f$ is a positive smooth function on $S^{n-1}$. Then there is a positive constant $C$ independent of $t$ such that
\begin{align}\label{eq401}
\frac{1}{C}\leq h(x,t)\leq C, \ \ \forall(x,t)\in S^{n-1}\times[0,T),
\end{align}
\begin{align}\label{eq402}
\frac{1}{C}\leq \rho(v,t)\leq C, \ \ \forall(v,t)\in S^{n-1}\times[0,T).
\end{align}
Here, $h(x,t)$ and $\rho(v,t)$ are the support function and the radial function of $\Omega_t$, respectively.
\end{lemma}
\begin{proof}
Due to  $\rho(v,t)v=\nabla h(x,t)+h(x,t)x$. Clearly, one sees
\begin{align*}
\min_{S^{n-1}} h(x,t)\leq \rho (v,t)\leq \max_{S^{n-1}} h(x,t).
\end{align*}
This implies that the estimate (\ref{eq401}) is tantamount to the estimate (\ref{eq402}). Thus, we only need to estimate (\ref{eq401}) or (\ref{eq402}).

Firstly, we prove the uniform  lower bound of $h(x,t)$. Here, we denote a concave increasing function $\mathcal{F}(y),y\in(0,a](a>0)$ and $\lim_{\epsilon\rightarrow 0}\int_\epsilon^a\mathcal{F}(y)dy=-\infty$ (for example, $\mathcal{F}(y)=-\frac{1}{y^n}$). Now, we construct a new auxiliary function
\begin{align*}
\Phi_1(\Omega_t)=c+\mathcal{F}\bigg(\int_{S^{n-1}}h(x,t)f(x)dx\bigg),
\end{align*} 
where $c$ is a positive constant (independent of $t$) such that 
\begin{align*}
\Phi_1(\Omega_0)=c+\mathcal{F}\bigg(\int_{S^{n-1}}h(x,0)f(x)dx\bigg)=c_0>0.
\end{align*}
Then
\begin{align*}
\frac{\partial}{\partial t}\Phi_1(\Omega_t)
=\mathcal{F}^{\prime}\bigg(\int_{S^{n-1}}h(x,t)f(x)dx\bigg)\int_{S^{n-1}}f(x)\frac{\partial h}{\partial t}dx.
\end{align*}
From the proof of Lemma \ref{lem33}, we know that $\Phi_1(\Omega_t)$ is also unchanged along the flow (\ref{eq110}) with $t\in[0,T)$. We now denote $h_{\min_t}(x)$ (or $h_{\inf_t}(x)$) by the minimum value or infimum of $h(x,t)$ $w.r.t.$ time $t$. Here, for convenience and without loss of generality, we assume $\int_{S^{n-1}}f(x)dx=1$, since $\mathcal{F}$ is a concave increasing function, then by Jensen inequality, there is
\begin{align*}
\Phi_1(\Omega_0)=\Phi_1(\Omega_t)=c+\mathcal{F}\bigg(\int_{S^{n-1}}h(x,t)f(x)dx\bigg)\geq c+\int_{S^{n-1}}\mathcal{F}(h(x,t))f(x)dx.
\end{align*}
Thus
\begin{align*}
\Phi_1(\Omega_0)^{-1}\leq&\frac{1}{c+\int_{S^{n-1}}\mathcal{F}(h(x,t))f(x)dx}\\
\leq&\frac{1}{c+\int_{S^{n-1}}\mathcal{F}(h_{\min_t}(x))f(x)dx}.
\end{align*}
The foregoing inequality shows that $h_{\min_t}(x)\nrightarrow0$ (or $h_{\inf_t}(x)\nrightarrow 0$) for any $x\in S^{n-1}$. In fact, if there is a point $x_0\in S^{n-1}$ such that $h_{\min_t}(x)\rightarrow 0$ when $x\rightarrow x_0$ and denote the neighborhood of $x_0$ being with $U(x_0,\epsilon)\subset S^{n-1}$ and the measure $|U(x_0,\epsilon)|>0$ for any small $\varepsilon$. Thus we have
\begin{align*}
\Phi_1(\Omega_0)^{-1}
\leq&\frac{1}{c+\min\limits_{S^{n-1}}f(x)\int_{U(x_0,\epsilon)}\mathcal{F}(h_{\min_t}(x))dx+\int_{S^{n-1}\setminus U(x_0,\epsilon)}\mathcal{F}(h_{\min_t}(x))f(x)dx}\\
\rightarrow&\frac{1}{c+(-\infty)+\int_{S^{n-1}\setminus U(x_0,\epsilon)}\mathcal{F}(h_{\min_t}(x))f(x)dx}=0^-
\end{align*}
as $x\rightarrow x_0$. This is a contradictory with $\Phi_1(\Omega_0)^{-1}=\frac{1}{c_0}>0$, namely, there is no point that makes $h_{\min_t}(x)\rightarrow 0$. Thus one can take $\delta$ independent of $t$ small enough to draw that $h_{\min_t}(x)\geq\delta$. The same discussion applies to $h_{\inf_t}(x)\geq\delta$. The support function for low one-dimensional
convex bodies can be similarly proven. At the same time, we can obtain $\Omega_t$ is a convex body containing the origin in its interior point, i.e. $\Omega_t\in\mathcal{K}_o^n$.

Next, we will derive at the uniform upper bound of $h(x,t)$. We has attained $\Omega_t\in\mathcal{K}_o^n$, from Lemma \ref{lem41}, then we can write
\begin{align*}
h(x,t)\geq& x\cdot x^t_{\max}h(x^t_{\max},t),\quad \forall x\in S^{n-1},
\end{align*}
where $x^t_{\max}$ is the point such that $h(x^t_{\max},t)=\max_{S^{n-1}}h(\cdot,t)$. Thus, from Lemma \ref{lem33}, we obtain
\begin{align*}
\Phi(\Omega_0)=\Phi(\Omega_t)=&\int_{S^{n-1}}h(x,t)f(x)dx\\
\geq &\int_{S^{n-1}}f(x)[h(x^t_{\max},t)x\cdot x^t_{\max}]dx\\
\geq&\frac{1}{2}h(x^t_{\max},t)\int_{\{x\in S^{n-1}: x\cdot x^t_{\max}\geq\frac{1}{2}\}}f(x)dx\\
\geq&Ch(x^t_{\max},t).
\end{align*}
This yields
\begin{align*}
\sup h(x^t_{\max},t)\leq \frac{\Phi(\Omega_0)}{C},
\end{align*}
where $C$ is a positive constant independent of $t$.
\end{proof}

\begin{lemma}\label{lem43}Let $\partial\Omega_t$ be a smooth strictly convex solution to the flow (\ref{eq110}) in $\mathbb{R}^n$ and $u(X,t)$ be the smooth admissible solution of (\ref{eq101}) in $\Omega_t$, $f$ is a positive smooth function on $S^{n-1}$. Then there is a positive constant $C$ independent of $t$ such that
\begin{align}\label{eq403}|\nabla h(x,t)|\leq C,\quad\forall(x,t)\in S^{n-1}\times [0,T),\end{align}
and
\begin{align}\label{eq404}|\nabla \rho(v,t)|\leq C,\quad \forall(v,t)\in S^{n-1}\times [0,T).\end{align}
\end{lemma}

\begin{proof}
The desired results immediately follow from Lemma \ref{lem42} and the following identities (see e.g. \cite{LR})
\begin{align*}
h=\frac{\rho^2}{\sqrt{\rho^2+|\nabla\rho|^2}},\qquad\rho^2=h^2+|\nabla h|^2.\end{align*}
\end{proof}

\begin{lemma}\label{lem44}Let $\partial\Omega_t$ be a smooth strictly convex solution to the flow (\ref{eq110}) in $\mathbb{R}^n$ and $u(X,t)$ be the smooth admissible solution of (\ref{eq101}) in $\Omega_t$, $f$ is a positive smooth function on $S^{n-1}$. Then there is a positive constant $C$ independent of $t$ such that
\begin{align*}\frac{1}{C}\leq\eta(t)\leq C.\end{align*}
\end{lemma}

\begin{proof}
From the definition of $\eta(t)$, Lemma \ref{lem32} and Lemma \ref{lem33}, we can directly obtain the lower bound of $\eta(t)$, namely,
\begin{align*}
\eta(t)=\frac{\int_{S^{n-1}}h(x,t)\sigma_{n-k}(x,t)|D u(X,t)|^{k+1}dx}{\int_{S^{n-1}}h(x,t)f(x)dx}=\frac{k(n+2)\widetilde{T}_k(\Omega_t)}{\Phi(\Omega_t)}\geq\frac{k(n+2)\widetilde{T}_k(\Omega_0)}{\Phi(\Omega_0)}.
\end{align*}

Since $\Omega_t$ is a smooth strictly convex body for any $t\in[0,T)$ and we have obtained uniform upper bound and uniform lower bound of $\Omega_t$ in Lemma \ref{lem42}. Thus, there exist balls $B_R$ and $B_r$ with radii of $R\leq R_0<\infty$ and $r\geq \delta>0$ such that $B_r\subset \Omega_t\subset B_R$, for balls $B_R$ and $B_r$, we have for any $x\in S^{n-1}$
\begin{align*}
\left\{
\begin{array}{lc}
S_k(D^2u_R(X(x)))=1\ \ \  \text{in} \ \ \ B_R,\\
u_R=0,\ \ \ \ \ \ \ \ \text{on} \ \ \ \ \partial B_R,\\
\end{array}
\right.
\end{align*}
and
\begin{align*}
\left\{
\begin{array}{lc}
S_k(D^2u_r(X(x)))=1\ \ \  \text{in} \ \ \ B_r,\\
u_r=0,\ \ \ \ \ \ \ \ \text{on} \ \ \ \ \partial B_r,\\
\end{array}
\right.
\end{align*}
The analysis of radial symmetric solutions provides an expression for the explicit solution \cite{NMA}, for example, for ball $B_R$, $u_R=c_{n,k}(R^2-|X(x)|^2)$, where $c_{n,k}$ depends on dimension $n$ and $k$, then $|Du_R|=2c_{n,k}R$. Similarly, $|Du_r|=2c_{n,k}r$.

Since $B_r\subset \Omega_t\subset B_R$ and $u=0$ on $\partial\Omega_t$, moreover, $u$ is a smooth admissible solution of (\ref{eq101}) on $\Omega_t$. For any $x\in S^{n-1}$ and $t\in[0,T)$, any point $X(x)\in\partial\Omega_t$, there exists ball $B_r$ such that $B_r\subset\Omega_t$ and $\partial\Omega_t\cap B_r=X(x)$. Because of the same equation and $u(\cdot,t)\geq 0=u_r(\cdot)$ on $\partial B_r$, hence using the maximum principle of $k$-Hessian equation \cite{CA}, we can obtain $u(\cdot,t)\geq u_r(\cdot)$ in $B_r$ and $u(X(x,t),t)=u_r(X(x))$, we have $|Du(X(x,t),t)|\geq|Du_r(X(x))|$. Similarly, we attain the upper bound $|Du(X(x,t),t)|\leq |Du_R(X(x))|$ by comparing it with $u_R$. Thus, we obtain
\begin{align*}
c\delta\leq |Du(X(x,t),t)|\leq CR_0,
\end{align*}
where $c$ and $C$ independent of $t$.

Consequently, by virtue of the uniform upper bounds of $h(x, t)$ and $|Du(X(x,t),t)|$, we arrive at immediately the upper bound of $\eta(t)$.
\end{proof}

\subsection{$C^2$ estimate}
In this subsection, we establish the upper bound and the lower bound of principal curvature. This will show that the equation (\ref{eq301}) is uniformly parabolic. Let us first state the following lemma before the $C^2$ estimate.

\begin{lemma}\label{lem45} Let $1\leq k\leq n-1$, $\Omega_t$ be a convex body of $C_+^2$ in $\mathbb{R}^n$ and $u(X(x,t),t)$ be the solution of (\ref{eq101}) in $\Omega_t$, then
\begin{flalign*}
\begin{split}
(i)&(D^2u(X(x,t),t)e_i)\cdot e_j=-\frac{1}{\sigma_{n-k}}|D u(X(x,t),t)|d_{ij}(x,t);\\
(ii)&(D^2u(X(x,t),t)e_i)\cdot x=-\frac{1}{\sigma_{n-k}}|D u(X(x,t),t)|_jd_{ij}(x,t);\\
(iii)&(D^2u(X(x,t),t)x)\cdot x=\frac{(n-k)|Du(X(x,t),t)|}{\sigma_{n-k}}.
\end{split}&
\end{flalign*}
Here, $e_i$ and $x$ are the orthogonal frame and the unite outer normal on $S^{n-1}$, $\cdot$ is the standard inner product and $d_{ij}$ is the cofactor matrix of $(h_{ij}+h\delta_{ij})$ with $\sum_{i,j}d_{ij}(h_{ij}+h\delta_{ij})=(n-k)\sigma_{n-k}$.
\end{lemma}

\begin{proof} (i)~~Assume that $h(x,t)$ is the support function of $\Omega_t$ for $(x,t)\in S^{n-1}\times (0,\infty)$ and let $\iota=\frac{\partial h}{\partial t}$. Then $X(x,t)=h_ie_i+hx, \frac{\partial X(x,t)}{\partial t}=\dot{X}(x,t)=\frac{\partial}{\partial t}(h_ie_i+hx)=\iota_ie_i+\iota x$. $X_i(x,t)=(h_{ij}+h\delta_{ij})e_j$, let $h_{ij}+h\delta_{ij}=\omega_{ij}$, then $X_{ij}(x,t)=\omega_{ijk}e_k-\omega_{ij}x$, where $\omega_{ijk}$ is the covariant derivatives of $\omega_{ij}$.

From $u(X,t)=0$ on $\partial\Omega_t$, we can not difficult to obtain
\begin{align*}
Du\cdot X_i=0,
\end{align*}
and
\begin{align*}((D^2u)X_j)X_i+Du X_{ij}=0.\end{align*}
It follows that
\begin{align}\label{eq405}
\omega_{ik}\omega_{jl}(((D^2u)e_l)\cdot e_k)+\omega_{ij}|Du|=0.
\end{align}
Multiplying both sides of (\ref{eq405}) by $d_{ij}$, we have
\begin{align*}
d_{ij}\omega_{ik}\omega_{jl}(((D^2 u)e_l)\cdot e_k)+\sigma_{n-k}|Du|=0.
\end{align*}
Namely,
\begin{align*}
\delta_{jk}\sigma_{n-k}\omega_{jl}(((D^2 u)e_l)\cdot e_k)+\sigma_{n-k}|Du|=0.
\end{align*}
It yields
\begin{align*}
\omega_{ij}(((D^2 u)e_i)\cdot e_j)+|Du|=0,
\end{align*}
then
\begin{align*}
d_{ij}\omega_{ij}(((D^2 u)e_i)\cdot e_j)+d_{ij}|Du|=0,
\end{align*}
i.e.
\begin{align*}
\sigma_{n-k}(((D^2 u)e_i)\cdot e_j)+d_{ij}|Du|=0,
\end{align*}
thus,
\begin{align*}
((D^2 u)e_i)\cdot e_j=-\frac{1}{\sigma_{n-k}}d_{ij}|Du|.
\end{align*}
This gives proof of (i).

(ii)~~Recall that
\begin{align*}
|Du(X(x,t),t)|=-Du(X(x,t),t)\cdot x,
\end{align*}
taking the covariant of both sides for above formula, we obtain
\begin{align}\label{eq406}
|Du|_j=-Du\cdot e_j-(D^2 u)X_j\cdot x=-\omega_{ij}((D^2 u)e_i\cdot x).
\end{align}
Multiplying both sides of (\ref{eq406}) by $d_{lj}$ and combining
\begin{align*}
d_{lj}\omega_{ij}=\delta_{li}\sigma_{n-k}.
\end{align*}
We conclude that
\begin{align*}
d_{ij}|Du|_j=-\sigma_{n-k}(D^2u)e_i\cdot x.
\end{align*}
Hence,
\begin{align*}
((D^2u)e_i)\cdot x=-\frac{1}{\sigma_{n-k}}d_{ij}|Du|_j.
\end{align*}
This proves (ii).

(iii)~~Hessian matrix $D^2u$ can be decomposed into
\begin{align*}
D^2u=\sum_{i,j=1}^{n-1}u_{ij}e_i\otimes e_j+u_{xx}x\otimes x,
\end{align*}
where tangential component $u_{ij}=(D^2ue_i)\cdot e_j=-\frac{|Du|}{\sigma_{n-k}}d_{ij}$ from (i), and normal component $u_{xx}=(D^2ux)\cdot x$.

From (\ref{eq101}), we know that
\begin{align*}
1=S_k(D^2u)=\frac{1}{k}S_k^{ij}(D^2u)u_{ij},
\end{align*}
where $S_k^{ij}$ is the partial derivative of $S_k$ about $u_{ij}$ (elements of the cofactor matrix), and
\begin{align*}
S_k^{ij}(D^2u)u_{ij}=S_k^{xx}(D^2u)u_{xx}+\sum_{i,j=1}^{n-1}S_k^{ij}(D^2u)u_{ij}.
\end{align*}
Since $S_k^{xx}(D^2u)$ is the partial derivative of $S_k$ about $u_{xx}$, and $S_k$ is linear about $u_{xx}$ (because $S_k$ is a symmetric function of eigenvalues), then
\begin{align*}
S_k^{xx}(D^2u)=S_{k-1}\text{(tangential part)},
\end{align*}
similarly,
\begin{align*}
S_k^{ij}(D^2u)=\frac{S_k}{\partial u_{ij}}=\text{elements of cofactor matrix}.
\end{align*}

\begin{align*}
D^2_{\tan}u=-\frac{|Du|}{\sigma_{n-k}}\sum_{ij}d_{ij}e_i\otimes e_j,
\end{align*}
\begin{align*}
S_l(\text{tangential part})=-\frac{(-1)^l|Du|^l}{\sigma^l_{n-k}}\sum_{ij}S_l(d_{ij}).
\end{align*}
Notice that $d_{ij}$ is the cofactor matrix of $(h_{ij}+h\delta_{ij})$ and
\begin{align*}
\sum_{i,j}d_{ij}(h_{ij}+h\delta_{ij})=(n-k)\sigma_{n-k},
\end{align*}
and
\begin{align*}
S_{k-1}(d_{ij})=(n-k)\sigma_{n-k-1},\quad S_k(d_{ij})=\sigma_{n-k}.
\end{align*}
Then
\begin{align*}
1=\frac{1}{k}[S_{k-1}(\text{tangential part})u_{xx}+S_k(\text{tangential part})].
\end{align*}
Computing
\begin{align*}
S_{k-1}(\text{tangential part})=\frac{(-1)^{k-1}|Du|^{k-1}}{\sigma_{n-k}^{k-1}}(n-k)\sigma_{n-k-1},
\end{align*}
and
\begin{align*}
S_k(\text{tangential part})=\frac{(-1)^k|Du|^k}{\sigma_{n-k}^k}\sigma_{n-k}=\frac{(-1)^k|Du|^k}{\sigma_{n-k}^{k-1}}.
\end{align*}
Thus,
\begin{align*}
1=\frac{1}{k}\bigg(\frac{(-1)^{k-1}(n-k)\sigma_{n-k-1}|Du|^{k-1}}{\sigma_{n-k}^{k-1}}u_{xx}+\frac{(-1)^k|Du|^k}{\sigma_{n-k}^{k-1}}\bigg),
\end{align*}
namely,
\begin{align*}
1=\frac{(-1)^{k-1}(n-k)\sigma_{n-k-1}|Du|^{k-1}}{k\sigma_{n-k}^{k-1}}u_{xx}+\frac{(-1)^k|Du|^k}{\sigma_{n-k}^{k-1}}.
\end{align*}
Note that
\begin{align*}
\frac{(-1)^{k-1}(n-k)\sigma_{n-k-1}|Du|^{k-1}}{k\sigma_{n-k}^{k-1}}u_{xx}=1-\frac{(-1)^k|Du|^k}{\sigma_{n-k}^{k-1}}.
\end{align*}
Suppose $k$ is a odd number (similar verifiable even numbers), then
\begin{align*}
u_{xx}=\frac{(n-k)|Du|}{\sigma_{n-k}}.
\end{align*}
Hence, the normal Hessian component is
\begin{align*}
(D^2u(X(x,t),t)x)\cdot x=u_{xx}=\frac{(n-k)|Du(X(x,t),t)|}{\sigma_{n-k}}.
\end{align*}
This completes proof of (iii).
\end{proof}

Now, we establish the lower bound of $\sigma_{n-k}(x,t)$.
\begin{lemma}\label{lem46} Under the conditions of Lemma \ref{lem42}, then there is a positive constant $C_0$ independent of $t$ such that
\begin{align*}
\sigma_{n-k}\geq C_0.
\end{align*}
\end{lemma}

\begin{proof}Consider a following auxiliary function, which  first appeared in \cite{KH},
\begin{align*}
E=&\log\bigg(\frac{h}{f(x)}|Du|^{k+1}\sigma_{n-k}\bigg)-A\frac{\rho^2}{2},
\end{align*}
where $A$ is a positive constant which will be chosen later.

Denote $\frac{h}{f(x)}|Du|^{k+1}\sigma_{n-k}=G\sigma_{n-k}=F$ and $\frac{\partial h(x,t)}{\partial t}=h_t$, then $h_t=F-\eta(t)h$ by (\ref{eq301}). Thus, the evolution equation of $E$ is written as
\begin{align*}
\frac{\partial E}{\partial t}=&\frac{1}{F}\frac{\partial F}{\partial t}-A\frac{\partial(\frac{\rho^2}{2})}{\partial t}.
\end{align*}
We compute the evolution equation of $F$,
\begin{align*}
\frac{\partial F}{\partial t}=\sigma_{n-k}\frac{\partial G}{\partial t}+G\frac{\partial\sigma_{n-k}}{\partial t},
\end{align*}
where
\begin{align*}
\frac{\partial G}{\partial t}=\frac{1}{f(x)}\bigg(|Du|^{k+1}h_t+(k+1)h|Du|^k\frac{\partial |Du|}{\partial t}\bigg).
\end{align*}

Since $|Du(X(x,t),t)|=-\langle Du(X(x,t),t),x\rangle$, $X(x,t)=h_ke_k+hx$, then
\begin{align}\label{eq407}
\frac{\partial |Du(X(x,t),t)|}{\partial t}=-[\langle (D^2u)x,(h_{tk}e_k+h_tx)\rangle+\langle D\dot{u},x\rangle].
\end{align}
From $u(X(x,t),t)=0$ on $\partial\Omega_t$, taking the derivative of both sides with respect to $t$, then we obtain
\begin{align*}
\dot{u}+Du \frac{\partial X(x,t)}{\partial t}=0,
\end{align*}
thus,
\begin{align}\label{eq408}
\dot{u}=-Du\cdot (h_{tk}e_k+h_tx)=|Du|x\cdot (h_{tk}e_k+h_tx)=|Du|h_t(x).
\end{align}
From (\ref{eq408}), we further calculate
\begin{align}\label{eq409}
\nonumber\langle D\dot{u},x\rangle=&\langle D(|Du|h_t),x\rangle=(\langle |Du|^{-1}Du D^2u,x\rangle)h_t+\langle|Du|(\nabla h)_t,x\rangle\\
\nonumber=&(\langle |Du|^{-1}Du D^2u,x\rangle)h_t+\langle|Du|(h_ke_k+hx)_t,x\rangle\\
=&(\langle |Du|^{-1}Du D^2u,x\rangle)h_t+|Du|h_t.
\end{align}
Substituting (\ref{eq409}) into (\ref{eq407}), we obtain
\begin{align}\label{eq410}
\nonumber&\frac{\partial |Du(X(x,t),t)|}{\partial t}\\
\nonumber=&-h_{tk}\langle (D^2u)x,e_k\rangle-h_t \langle(D^2u)x,x\rangle -(\langle |Du|^{-1}Du D^2u,x\rangle)h_t-|Du|h_t\\
=&-(h_t)_k\langle (D^2u)x,e_k\rangle-\bigg(\langle(D^2u)x,x\rangle +(\langle |Du|^{-1}Du D^2u,x\rangle)+|Du|\bigg)h_t.
\end{align}
Thus, combining (\ref{eq410}), we obtain
\begin{align}\label{eq411}
&\frac{\partial G}{\partial t}\sigma_{n-k}\\
\nonumber=&\frac{\sigma_{n-k}}{f(x)}\bigg\{|Du|^{k+1}h_t\\
\nonumber&+(k\!+\!1)h|Du|^k\bigg[\!-\!(h_t)_k\langle (D^2u)x,e_k\rangle-\bigg(\langle(D^2u)x,x\rangle+(\langle |Du|^{-1}Du D^2u,x\rangle)+|Du|\bigg)h_t\!\bigg]\!\bigg\}.
\end{align}

Recall $\sigma_{n-k}(x,t)=\sigma_{n-k}(\omega_{ij}(x,t))$, $\omega_{ij}(x,t)=h_{ij}(x,t)+h(x,t)\delta_{ij}$ and $d_{ij}=\frac{\partial \sigma_{n-k}}{\partial \omega_{ij}}$, then
\begin{align}\label{eq412}
\nonumber\frac{\partial \sigma_{n-k}}{\partial t}=&d_{ij}\frac{\partial \omega_{ij}}{\partial t}=d_{ij}\nabla_{ij}(h_t)+d_{ij}h_t\delta_{ij}\\
\nonumber=&d_{ij}\nabla_{ij}(F-\eta(t)h)+d_{ij}\delta_{ij}(F-\eta(t)h)\\
\nonumber=&d_{ij}F_{ij}-\eta(t)d_{ij}h_{ij}+d_{ij}\delta_{ij}(F-\eta(t)h)\\
\nonumber=&d_{ij}F_{ij}+Fd_{ij}\delta_{ij}-\eta(t)d_{ij}(h_{ij}+h\delta_{ij})\\
=&d_{ij}F_{ij}+Fd_{ij}\delta_{ij}-(n-k)\eta(t)\sigma_{n-k},
\end{align}
where we use the $(n-k)$-degree homogeneity of $\sigma_{n-k}$ in the last equality and obtain $d_{ij}\omega_{ij}=(n-k)\sigma_{n-k}$.

We know that $\rho^2=h^2+|\nabla h|^2$, thus,
\begin{align}\label{eq413}
\nonumber\frac{\partial(\frac{\rho^2}{2})}{\partial t}=&\frac{1}{2}\frac{\partial(h^2+|\nabla h|^2)}{\partial t}=hh_t+\sum h_kh_{kt}\\
=&h(F-\eta(t)h)+\sum h_k(F_k-\eta(t)h_k)=hF+\sum h_kF_k-\eta(t)\rho^2.
\end{align}

Combining (\ref{eq411}), (\ref{eq412}) and (\ref{eq413}), we get
\begin{align*}
\frac{\partial E}{\partial t}=&\frac{1}{F}\bigg\{\frac{\sigma_{n-k}}{f(x)}\bigg[|Du|^{k+1}h_t\\
&+(k\!+\!1)h|Du|^k\bigg(\!-\!(h_t)_k\langle (D^2u)x,e_k\rangle\!-\!\bigg(\!\langle(D^2u)x,x\rangle \!+\!(\langle |Du|^{-1}Du D^2u,x\rangle)\!+\!|Du|\bigg)h_t\!\bigg)\!\bigg]\\
&+G\bigg(d_{ij}F_{ij}+Fd_{ij}\delta_{ij}-(n-k)\eta(t)\sigma_{n-k}\bigg)\bigg\}-A[hF+\sum h_kF_k-\eta(t)\rho^2]\\
=&\frac{\sigma_{n-k}}{f(x)}\bigg[\frac{|Du|^{k+1}(F-\eta(t)h)}{F}-\frac{(k+1)h|Du|^k((F-\eta(t)h))_k\langle (D^2u)x,e_k\rangle}{F}\\
&-\frac{(k+1)h|Du|^k\bigg(\langle(D^2u)x,x\rangle +(\langle |Du|^{-1}Du D^2u,x\rangle)+|Du|\bigg)(F-\eta(t)h)}{F}\bigg]\\
&+\frac{G}{F}\bigg(d_{ij}F_{ij}+Fd_{ij}\delta_{ij}-(n-k)\eta(t)\sigma_{n-k}\bigg)-A[hF+\sum h_kF_k-\eta(t)\rho^2].
\end{align*}
Suppose the spatial minimum of $E$ is attained at a point $(x_0,t)$, then $F_k=0$, $F_{ij}\geq 0$, thus the following result is obtained by using Lemma \ref{lem45}, dropping some positive terms and rearranging terms
\begin{align*}
\frac{\partial E}{\partial t}\geq &\frac{\sigma_{n-k}}{f(x)}\bigg[\frac{|Du|^{k+1}(e^{E+A\frac{\rho^2}{2}}-\eta(t)h)}{e^{E+A\frac{\rho^2}{2}}}+\frac{(k+1)\eta(t)h|Du|^kh_k\langle (D^2u)x,e_k\rangle}{e^{E+A\frac{\rho^2}{2}}}\\
&-\frac{(k+1)h|Du|^k\bigg(\langle(D^2u)x,x\rangle +(\langle |Du|^{-1}Du D^2u,x\rangle)+|Du|\bigg)(e^{E+A\frac{\rho^2}{2}}-\eta(t)h)}{e^{E+A\frac{\rho^2}{2}}}\bigg]\\
&-(n-k)\eta(t)+A\frac{\rho^2}{2}-A\bigg(he^{E+A\frac{\rho^2}{2}}-\frac{\eta(t)\rho^2}{2}\bigg)\\
\geq&\frac{\sigma_{n-k}}{f(x)}\frac{e^{E+A\frac{\rho^2}{2}}-\eta(t)h}{e^{E+A\frac{\rho^2}{2}}}\bigg[|Du|^{k+1}h(1-(k+1))-2(k+1)h|Du|^k|D^2u|\bigg]\\
&+\frac{\eta(t)\rho^2}{2}\bigg(A-\frac{2(n-k)}{\rho^2}\bigg)+A\bigg(\frac{\eta(t)\rho^2}{2}-he^{E+A\frac{\rho^2}{2}}\bigg).
\end{align*}

We have obtained the uniform bound of $|Du(X(x,t),t)|$ in proof Lemma \ref{lem44}, by virtue of Schauder's theory (see example Chapter 17 in \cite{GI}), thus, we are easy to obtain $|D^2u(X(x,t),t)|\leq \widehat{C}$ on $S^{n-1}\times [0,T)$.

Now, choose $A>\max\frac{2(n-k)}{\rho^2}$. Thus, if $E$ becomes very negative, denote
\begin{align*}
L_1=\frac{e^{E+A\frac{\rho^2}{2}}-\eta(t)h}{e^{E+A\frac{\rho^2}{2}}}<0,
\end{align*}
since $-k<0$, then
\begin{align*}
L_2=|Du|^{k+1}h(1-(k+1))-2(k+1)h|Du|^k|D^2u|<0,
\end{align*}
\begin{align*}
L_3=+\frac{\eta(t)\rho^2}{2}\bigg(A-\frac{2(n-k)}{\rho^2}\bigg)>0,
\end{align*}
\begin{align*}
L_4=A\bigg(\frac{\eta(t)\rho^2}{2}-he^{E+A\frac{\rho^2}{2}}\bigg)>0,
\end{align*}
Then 
\begin{align*}
\frac{\partial E}{\partial t}\geq \frac{\sigma_{n-k}}{f(x)}L_1L_2+L_3+L_3>0.
\end{align*} 
Thus $E$ has the lower bound. Therefore we obtain the lower bound of $\sigma_{n-k}$.
\end{proof}

We will establish the upper bound of $\sigma_{n-k}$ in the next lemma. The proof needs a property of $\sigma_{n-k}$, this is
$d_{ij}\omega_{im}\omega_{jm}\geq (n-k)(\sigma_{n-k})^{1+\frac{1}{n-k}}$. (see \cite{AB0} for details.)

\begin{lemma}\label{lem47} Under the conditions of Lemma \ref{lem42}, then there is a positive constant $C_1$ independent of $t$ such that
\begin{align*}
\sigma_{n-k}\leq C_1.
\end{align*}
\end{lemma}
\begin{proof}
Considering an auxiliary function which was first introduced by Ivaki \cite{IV} and subsequently applied to Orlicz-Minkowski flows \cite{BR1}.

\begin{align*}
M=\frac{1}{1-\beta\frac{\rho^2}{2}}\frac{G\sigma_{n-k}}{h},
\end{align*}
where $G=\frac{h|Du|^{k+1}}{f(x)}$ and $\beta$ is a positive constant such that $2\beta\leq\rho^2\leq\frac{2}{\beta}$ for all $t\in[0,T)$ (know from Lemma \ref{lem42}). Suppose $(x_1,t)$ is a  spatial maximum value point of $M$. Then at point $(x_1,t)$,
\begin{align}\label{eq414}
\nabla_iM=0, \quad \text{i.e.},\quad \frac{\beta}{1-\beta\frac{\rho^2}{2}}\nabla_i\bigg(\frac{\rho^2}{2}\bigg)\frac{G\sigma_{n-k}}{h}+\nabla_i\bigg(\frac{G\sigma_{n-k}}{h}\bigg)=0,
\end{align}
and
\begin{align}\label{eq415}
\nabla_{ij}M\leq 0.
\end{align}
Now, we estimate $M$, from (\ref{eq415}), we obtain
\begin{align}\label{eq416}
\frac{\partial M}{\partial t}\leq &\frac{\partial M}{\partial t}-Gd_{ij}\nabla_{ij}M\\
\nonumber=&\frac{\partial \bigg(\frac{1}{1-\beta\frac{\rho^2}{2}}\frac{G\sigma_{n-k}}{h}\bigg)}{\partial t}-Gd_{ij}\nabla_{ij}\bigg(\frac{1}{1-\beta\frac{\rho^2}{2}}\frac{G\sigma_{n-k}}{h}\bigg)\\
\nonumber=&\frac{1}{1-\beta\frac{\rho^2}{2}}\bigg[\frac{\partial\bigg(\frac{G\sigma_{n-k}}{h}\bigg)}{\partial t}-Gd_{ij}\nabla_{ij}\bigg(\frac{G\sigma_{n-k}}{h}\bigg)\bigg]\\
\nonumber&+\frac{\beta}{\bigg(1-\beta\frac{\rho^2}{2}\bigg)^2}\frac{G\sigma_{n-k}}{h}\bigg[\frac{\partial(\frac{\rho^2}{2})}{\partial t}-Gd_{ij}\nabla_{ij}\bigg(\frac{\rho^2}{2}\bigg)\bigg]\\
\nonumber&-2Gd_{ij}\frac{\beta}{\bigg(1-\beta\frac{\rho^2}{2}\bigg)^2}\nabla_j\bigg(\frac{G\sigma_{n-k}}{h}\bigg)\nabla_i\bigg(\frac{\rho^2}{2}\bigg)\\
\nonumber&-2Gd_{ij}\frac{\beta^2}{\bigg(1-\beta\frac{\rho^2}{2}\bigg)^3}\frac{G\sigma_{n-k}}{h}\nabla_i\bigg(\frac{\rho^2}{2}\bigg)\nabla_j\bigg(\frac{\rho^2}{2}\bigg)\\
\nonumber=&\frac{1}{1-\beta\frac{\rho^2}{2}}\bigg[\frac{\partial\bigg(\frac{G\sigma_{n-k}}{h}\bigg)}{\partial t}-Gd_{ij}\nabla_{ij}\bigg(\frac{G\sigma_{n-k}}{h}\bigg)\bigg]\\
\nonumber&+\frac{\beta}{\bigg(1-\beta\frac{\rho^2}{2}\bigg)^2}\frac{G\sigma_{n-k}}{h}\bigg[\frac{\partial(\frac{\rho^2}{2})}{\partial t}-Gd_{ij}\nabla_{ij}\bigg(\frac{\rho^2}{2}\bigg)\bigg]\\
\nonumber&-2Gd_{ij}\frac{\beta}{\bigg(1-\beta\frac{\rho^2}{2}\bigg)^2}\nabla_i\bigg(\frac{\rho^2}{2}\bigg)\bigg[\nabla_j\bigg(\frac{G\sigma_{n-k}}{h}\bigg)
+\frac{\beta}{1-\beta\frac{\rho^2}{2}}\frac{G\sigma_{n-k}}{h}\nabla_j\bigg(\frac{\rho^2}{2}\bigg)\bigg].
\end{align}
From (\ref{eq414}), we can simplify (\ref{eq416}) to
\begin{align}\label{eq417}
\frac{\partial M}{\partial t}\leq &\frac{1}{1-\beta\frac{\rho^2}{2}}\bigg[\frac{\partial\bigg(\frac{G\sigma_{n-k}}{h}\bigg)}{\partial t}-Gd_{ij}\nabla_{ij}\bigg(\frac{G\sigma_{n-k}}{h}\bigg)\bigg]\\
\nonumber&+\frac{\beta}{\bigg(1-\beta\frac{\rho^2}{2}\bigg)^2}\frac{G\sigma_{n-k}}{h}\bigg[\frac{\partial(\frac{\rho^2}{2})}{\partial t}-Gd_{ij}\nabla_{ij}\bigg(\frac{\rho^2}{2}\bigg)\bigg].
\end{align}
Now, we calculate
\begin{align*}
&\frac{\partial\bigg(\frac{G\sigma_{n-k}}{h}\bigg)}{\partial t}-Gd_{ij}\nabla_{ij}\bigg(\frac{G\sigma_{n-k}}{h}\bigg)\\
=&\frac{\sigma_{n-k}\frac{\partial G}{\partial t}+G\frac{\partial \sigma_{n-k}}{\partial t}}{h}-\frac{G\sigma_{n-k}h_t}{h^2}-Gd_{ij}\frac{\nabla_{ij}[G\sigma_{n-k}]}{h}+Gd_{ij}\frac{G\sigma_{n-k}\nabla_{ij}h}{h^2}\\
&+2Gd_{ij}\sigma_{n-k}\frac{\nabla_i[G\sigma_{n-k}]\nabla_jh}{h^2}-2Gd_{ij}\frac{[G\sigma_{n-k}]\nabla_ih\nabla_jh}{h^3}\\
=&\frac{\sigma_{n-k}\frac{\partial G}{\partial t}+Gd_{ij}[(G\sigma_{n-k}-\eta(t)h)_{ij}+h_t\delta_{ij}]}{h}-\frac{G\sigma_{n-k}h_t}{h^2}-Gd_{ij}\frac{\nabla_{ij}[G\sigma_{n-k}]}{h}\\
&+Gd_{ij}\frac{G\sigma_{n-k}\nabla_{ij}h}{h^2}-2Gd_{ij}\frac{[G\sigma_{n-k}]\nabla_ih\nabla_jh}{h^3}\\
=&\frac{\frac{\partial G}{\partial t}\sigma_{n-k}-Gd_{ij}[\eta(t)\omega_{ij}-G\sigma_{n-k}\delta_{ij}]}{h}-\frac{G\sigma_{n-k}(G\sigma_{n-k}-\eta(t)h)}{h^2}\\
&+Gd_{ij}\frac{G\sigma_{n-k}\nabla_{ij}h}{h^2}-2Gd_{ij}\frac{[G\sigma_{n-k}]\nabla_ih\nabla_jh}{h^3}\\
=&\frac{\frac{\partial G}{\partial t}\sigma_{n-k}-(n-k)\eta(t)G\sigma_{n-k}+G^2\sigma_{n-k}d_{ij}\delta_{ij}}{h}-\frac{(G\sigma_{n-k})^2}{h^2}+\frac{G\sigma_{n-k}\eta(t)}{h}\\
&+Gd_{ij}\frac{G\sigma_{n-k}\nabla_{ij}h}{h^2}-2Gd_{ij}\frac{[G\sigma_{n-k}]\nabla_ih\nabla_jh}{h^3}.
\end{align*}

From definition of $G$ and (\ref{eq410}), we know that
\begin{align*}
&\frac{\partial G}{\partial t}=\frac{1}{f(x)}\bigg(|Du|^{k+1}h_t+(k+1)h|Du|^k\frac{\partial|Du|}{\partial t}\bigg)\\
=&\frac{1}{f(x)}\bigg[|Du|^{k+1}(G\sigma_{n-k}-\eta(t)h)\\
&-(k+1)h|Du|^k\bigg((h_t)_k\langle (D^2u)x,e_k\rangle+\bigg(\langle(D^2u)x,x\rangle +(\langle |Du|^{-1}Du D^2u,x\rangle)+|Du|\bigg)h_t\bigg)\bigg].
\end{align*}

Recall that $\rho^2=h^2+|\nabla h|^2$, then
\begin{align*}
&\frac{\partial(\frac{\rho^2}{2})}{\partial t}-Gd_{ij}\nabla_{ij}\bigg(\frac{\rho^2}{2}\bigg)\\
=&hh_t+\nabla_mh\nabla_m(h_t)-Gd_{ij}\bigg(h\nabla_{ij}h+\nabla_ih\nabla_jh+\nabla_mh\nabla_j\nabla_{mi}h+\nabla_{mi}h\nabla_{mj}h\bigg)\\
=&h(G\sigma_{n-k}-\eta(t)h)+[\sigma_{n-k}\nabla_mG\nabla_mh+G\nabla_m\sigma_{n-k}\nabla_mh-\eta(t)|\nabla h|^2]-Gd_{ij}h(\omega_{ij}-h\delta_{ij})\\
&-Gd_{ij}\nabla_ih\nabla_jh-Gd_{ij}(\omega_{mij}-h_m\delta_{ij})\nabla_mh-Gd_{ij}(\omega_{mi}-h\delta_{mi})(\omega_{mj}-h\delta_{mj})\\
=&(n+1-k)hG\sigma_{n-k}-\eta(t)\rho^2+\sigma_{n-k}\nabla_mh\nabla_mG-Gd_{ij}\omega_{mi}\omega_{mj},
\end{align*}
where we use the Codazzi equation $\omega_{imj}=\omega_{ijm}$ and the $(n-k)$-homogeneity of $\sigma_{n-k}$ in the last equality. Here,
\begin{align*}
\nabla_mG=\frac{-f_m}{f^2}h|Du|^{k+1}+\frac{1}{f}h_m|Du|^{k+1}+(k+1)\frac{h}{f}|Du|^k|Du|_m,
\end{align*}
since $|Du|^2=DU\cdot Du$, then $\nabla_m|Du|^2=2\nabla_m Du\cdot Du$, thus
\begin{align*}\nabla_m|Du|=|Du|^{-1}\nabla_m Du\cdot Du\leq |D^2u|.\end{align*}
Hence, we can obtain $\nabla_mG\leq \widetilde{C}$.

Substitute the above calculations into (\ref{eq417}), we attain
\begin{align*}
&\frac{\partial M}{\partial t}\\
\leq &\frac{1}{1-\beta\frac{\rho^2}{2}}\bigg[\frac{\frac{\partial G}{\partial t}\sigma_{n-k}-(n-k)\eta(t)G\sigma_{n-k}+G^2\sigma_{n-k}d_{ij}\delta_{ij}}{h}-\frac{(G\sigma_{n-k})^2}{h^2}+\frac{G\sigma_{n-k}\eta(t)}{h}\\
&+Gd_{ij}\frac{G\sigma_{n-k}\nabla_{ij}h}{h^2}-2Gd_{ij}\frac{[G\sigma_{n-k}]\nabla_ih\nabla_jh}{h^3}\bigg]\\
&+\frac{\beta}{\bigg(1-\beta\frac{\rho^2}{2}\bigg)^2}\frac{G\sigma_{n-k}}{h}\bigg[(n+1-k)hG\sigma_{n-k}-\eta(t)\rho^2+\sigma_{n-k}\nabla_mh\nabla_mG-Gd_{ij}\omega_{mi}\omega_{mj}\bigg]\\
=&\frac{1}{1-\beta\frac{\rho^2}{2}}\bigg\{\frac{\sigma_{n-k}}{f(x)h}\bigg[|Du|^{k+1}(G\sigma_{n-k}-\eta(t)h)\\
&-(k+1)h|Du|^k\bigg((h_t)_k\langle (D^2u)x,e_k\rangle+\bigg(\langle(D^2u)x,x\rangle +(\langle |Du|^{-1}Du D^2u,x\rangle)+|Du|\bigg)h_t\bigg)\bigg]\\
&-\frac{(n-k)\eta(t)G\sigma_{n-k}}{h}+\frac{G^2\sigma_{n-k}d_{ij}\delta_{ij}}{h}-\frac{(G\sigma_{n-k})^2}{h^2}+\frac{G\sigma_{n-k}\eta(t)}{h}\\
&+Gd_{ij}\frac{G\sigma_{n-k}(\omega_{ij}-h\delta_{ij})}{h^2}-2Gd_{ij}\frac{[G\sigma_{n-k}]\nabla_ih\nabla_jh}{h^3}\bigg\}\\
&+\frac{\beta}{\bigg(1-\beta\frac{\rho^2}{2}\bigg)^2}\frac{G\sigma_{n-k}}{h}\bigg[(n+1-k)hG\sigma_{n-k}-\eta(t)\rho^2+\sigma_{n-k}\nabla_mh\nabla_mG-Gd_{ij}\omega_{mi}\omega_{mj}\bigg]\\
\leq&\frac{1}{1-\beta\frac{\rho^2}{2}}\bigg\{\frac{F}{Ghf(x)}\bigg[|Du|^{k+1}(F-\eta(t)h)-(k+1)h|Du|^k\bigg((G\sigma_{n-k}-\eta(t)h)_k\langle (D^2u)x,e_k\rangle\\
&+\bigg(\langle(D^2u)x,x\rangle +(\langle |Du|^{-1}Du D^2u,x\rangle)+|Du|\bigg)F-\eta(t)h\bigg)\bigg]\\
&-\frac{(n-k)\eta(t)F}{h}+\frac{FGd_{ij}\delta_{ij}}{h}-\frac{F^2}{h^2}+\frac{F\eta(t)}{h}+\frac{(n-k)FG}{h^2}-2Gd_{ij}\frac{F\nabla_ih\nabla_jh}{h^3}\bigg\}\\
&+\frac{\beta}{\bigg(1-\beta\frac{\rho^2}{2}\bigg)^2}\bigg[(n+1-k)F^2+\frac{F^2}{hG}\nabla_mh\nabla_mG-\frac{F(n-k)}{h}\bigg(\frac{F}{G}\bigg)^{1+\frac{1}{n-k}}\bigg]\\
\leq&\frac{1}{1-\beta\frac{\rho^2}{2}}\bigg\{\frac{F}{Ghf(x)}\bigg[F|Du|^{k+1}+\bigg(2|D^2u|+|Du|\bigg)F\bigg)\bigg]+\frac{FGd_{ii}}{h}+\frac{F\eta(t)}{h}
+\frac{(n-k)FG}{h^2}\bigg\}\\
&+\frac{\beta}{\bigg(1-\beta\frac{\rho^2}{2}\bigg)^2}\bigg[(n+1-k)F^2+\frac{F^2}{hG}\nabla_mh\nabla_mG-\frac{n-k}{hG^{1+\frac{1}{n-k}}}F^{2+\frac{1}{n-k}}\bigg]\\
=&\bigg(Gd_{ii}+\eta(t)+\frac{(n-k)G}{h}\bigg)\bigg[\frac{1}{1-\beta\frac{\rho^2}{2}}\frac{F}{h}\bigg]\\
&+\!\bigg\{\frac{h(1-\!\beta\frac{\rho^2}{2})}{Gf(x)}\bigg(\!|Du|^{k+1}\!+\!2|D^2u|\!+\!|Du|\!\bigg)\!+\!\beta h^2\bigg(\!(n+1-k)\!+\!\frac{\nabla_mh\nabla_mG}{hG}\!\bigg)\bigg\}\bigg[\frac{1}{1-\beta\frac{\rho^2}{2}}\frac{F}{h}\!\bigg]^2\\
&-\beta\bigg(1-\beta\frac{\rho^2}{2}\bigg)^\frac{1}{n-k}(n-k)\bigg(\frac{h}{G}\bigg)^{1+\frac{1}{n-k}}\bigg[\frac{1}{1-\beta\frac{\rho^2}{2}}\frac{F}{h}\bigg]^{2+\frac{1}{n-k}}.
\end{align*}
From (\ref{eq304+}) and $\Omega_t$ is a smooth strictly convex body with uniform bound, it's not difficult to see that $d_{ii}$ has uniform upper bound. Taking
\begin{align*}
P_1=&Gd_{ii}+\eta(t)+\frac{(n-k)G}{h}\leq C_2,\\
P_2=&\frac{h(1-\!\beta\frac{\rho^2}{2})}{Gf(x)}\bigg(\!|Du|^{k+1}\!+\!2|D^2u|\!+\!|Du|\!\bigg)\!+\!\beta h^2\bigg(\!(n+1-k)\!+\!\frac{\nabla_mh\nabla_mG}{hG}\!\bigg)\leq C_3,\\
P_3=&\beta\bigg(1-\beta\frac{\rho^2}{2}\bigg)^\frac{1}{n-k}(n-k)\bigg(\frac{h}{G}\bigg)^{1+\frac{1}{n-k}}\leq C_4.
\end{align*}

Thus, according to above calculations, at point $x_1$, there exist some positive constants $C_2$, $C_3$ and $C_4$ only depend on $f$ and $n$ such that
\begin{align*}
\frac{\partial M}{\partial t}\leq C_2M+C_3M^2-C_4M^{2+\frac{1}{n-k}}<0
\end{align*}
provided $M$ is sufficiently large. Thus, $M(x,t)$ is uniformly bounded from above. The upper bound of $\sigma_{n-k}$ follows from the uniformly bound of $M$.
\end{proof}

From \cite{UR}, we know that the eigenvalues of $\{\omega_{ij}\}$ and $\{\omega^{ij}\}$ are respectively the principal radii and principal curvatures
of $\Omega_t$, where $\{\omega^{ij}\}$ is the inverse matrix of $\{\omega_{ij}\}$. Therefore, to derive a positive
upper bound of principal curvatures of $\Omega_t$ at $X(x,t)$, it is equivalent to estimate the upper bound of the eigenvalues of $\{\omega^{ij}\}$.

\begin{lemma}\label{lem48} Under conditions of Lemma \ref{lem42}, there exists a positive constant $C$ independent of $t$ such that
\begin{align*}
\frac{1}{C}\leq\kappa_i(\cdot,t)\leq C,\qquad i=1,\cdots,n-1.
\end{align*}
\end{lemma}
\begin{proof} For any fixed $t\in[0,T)$, we suppose that the spatial maximum of eigenvalue of matrix $\{\frac{\omega^{ij}}{h}\}$ attained at a point $x_2$ in the direction of the unit vector $e_1\in T_{x_2}S^{n-1}$. By a rotation, we also choose the orthonormal vector field such that $\omega_{ij}$ is diagonal and the maximum eigenvalue of $\{\frac{\omega^{ij}}{h}\}$ is $\frac{\omega^{11}}{h}$.

Firstly, we calculate the evolution equation of $\omega_{ij}$ and $\omega^{ij}$. For convenience, we set $G=\frac{h|Du|^{k+1}}{f(x)}$, then $h_t=G\sigma_{n-k}-\eta(t)h$. Since $\omega_{ij}=\nabla_{ij}h+h\delta_{ij}$, we obtain
\begin{align*}
\frac{\partial \omega_{ij}}{\partial t}=&\nabla_{ij}(h_t)+h_t\delta_{ij}\\
=&\nabla_{ij}[G\sigma_{n-k}-\eta(t)h]+\bigg(G\sigma_{n-k}-\eta(t)h\bigg)\delta_{ij}\\
=&\sigma_{n-k}\nabla_{ij}G+\nabla_iG\nabla_j\sigma_{n-k}+\nabla_i\sigma_{n-k}\nabla_jG+G\nabla_{ij}\sigma_{n-k}+G\sigma_{n-k}\delta_{ij}-\eta(t)\omega_{ij},
\end{align*}
where
\begin{align*}
\nabla_i\sigma_{n-k}=d_{mn}\nabla_i(\omega_{mn}),
\end{align*}
and
\begin{align*}
\nabla_{ij}\sigma_{n-k}=d_{mn,ls}\nabla_j(\omega_{ls})\nabla_i(\omega_{mn})+d_{mn}\nabla_{ij}(\omega_{mn}).
\end{align*}
By the Codazzi equation and the Ricci identity, we have
\begin{align*}
d_{mn}\nabla_{ij}(\omega_{mn})=&d_{mn}\nabla_{nj}(\omega_{mi})\\
=&d_{mn}\nabla_{jn}(\omega_{mi})+d_{mn}\omega_{pm}\nabla_{nj}R_{pi}+d_{mn}\delta_{pi}\nabla_{nj}R_{pm}\\
=&d_{mn}\nabla_{mn}\omega_{ij}+d_{mn}\omega_{mn}\delta_{ij}-d_{mn}\omega_{jm}\delta_{in}+d_{mn}\omega_{in}\delta_{mj}-d_{mn}\omega_{ij}\delta_{mn}\\
=&d_{mn}\nabla_{mn}\omega_{ij}+(n-k)\sigma_{n-k}\delta_{ij}-d_{mn}\delta_{mn}\omega_{ij}.
\end{align*}
Then
\begin{align*}
&\frac{\partial \omega_{ij}}{\partial t}\\
=&\sigma_{n-k}\nabla_{ij}G+\nabla_iG\nabla_j\sigma_{n-k}+\nabla_i\sigma_{n-k}\nabla_jG+(n+1-k)G\sigma_{n-k}\delta_{ij}-\eta(t)\omega_{ij}\\
&+G\bigg(d_{mn,ls}\nabla_j(\omega_{ls})\nabla_i(\omega_{mn})+d_{mn}\nabla_{mn}\omega_{ij}-d_{mn}\delta_{mn}\omega_{ij}\bigg).
\end{align*}
Hence,
\begin{align}\label{eq418}
&\frac{\partial \omega_{ij}}{\partial t}-Gd_{mn}\nabla_{mn}\omega_{ij}\\
\nonumber=&\sigma_{n-k}\nabla_{ij}G+\nabla_iG\nabla_j\sigma_{n-k}+\nabla_i\sigma_{n-k}\nabla_jG+(n+1-k)G\sigma_{n-k}\delta_{ij}-\eta(t)\omega_{ij}\\
\nonumber&+G\bigg(d_{mn,ls}\nabla_j(\omega_{ls})\nabla_i(\omega_{mn})-d_{mn}\delta_{mn}\omega_{ij}\bigg).
\end{align}

Since $\frac{\partial \omega^{ij}}{\partial t}=-(\omega^{ij})^2\frac{\partial \omega_{ij}}{\partial t}$ and $\nabla_{mn}\omega^{ij}=2(\omega^{ij})^3\nabla_m\omega_{ij}\nabla_n\omega_{ij}-(\omega^{ij})^2\nabla_{mn}\omega_{ij}$, thus, there is the following evolution equation by (\ref{eq418}),
\begin{align}\label{eq419}
&\frac{\partial \omega^{ij}}{\partial t}-Gd_{mn}\nabla_{mn}\omega^{ij}\\
\nonumber=&-(\omega^{ij})^2\sigma_{n-k}\nabla_{ij}G-(\omega^{ij})^2\nabla_iG\nabla_j\sigma_{n-k}
-(\omega^{ij})^2\nabla_i\sigma_{n-k}\nabla_jG\\
\nonumber&-(n+1-k)(\omega^{ij})^2G\sigma_{n-k}\delta_{ij}+\eta(t)\omega^{ij}\\
\nonumber&-G(\omega^{ij})^2\bigg(d_{mn,ls}\nabla_j(\omega_{ls})\nabla_i(\omega_{mn})-d_{mn}\delta_{mn}\omega_{ij}\bigg)-2Gd_{mn}(\omega^{ij})^3\nabla_m\omega_{ij}\nabla_n\omega_{ij}.
\end{align}

At $x_2$, we get
\begin{align}\label{eq420}
\nabla_i\frac{\omega^{11}}{h}=0,\quad \text{i.e.},\quad \omega^{11}\nabla_i\omega_{11}=-\frac{\nabla_ih}{h},
\end{align}
\begin{align*}
\nabla_{ij}\omega_{11}=&\frac{\omega_{11}\nabla_ih \nabla_jh}{h^2}-\frac{\omega_{11}\nabla_{ij}h+\nabla_ih\nabla_j\omega_{11}}{h}\\
=&2\frac{\omega_{11}\nabla_ih \nabla_jh}{h^2}-\frac{\omega_{11}\nabla_{ij}h}{h}.
\end{align*}
And
\begin{align}\label{eq421}
\nabla_{ij}\frac{\omega^{11}}{h}\leq0.
\end{align}
Now, from (\ref{eq419}) and (\ref{eq421}), we compute the following evolution equation as
\begin{align}\label{eq422}
\frac{\partial(\frac{\omega^{11}}{h})}{\partial t}\leq&\frac{\partial(\frac{\omega^{11}}{h})}{\partial t}-Gd_{ij}\nabla_{ij}\frac{\omega^{11}}{h}\\
\nonumber=&\frac{\frac{\partial \omega^{11}}{\partial t}}{h}-\frac{\omega^{11}h_t}{h^2}-Gd_{ij}\bigg(\frac{2(\omega^{11})^3\nabla_i\omega_{11}\nabla_j\omega_{11}-(\omega^{11})^2\nabla_{ij}\omega_{11}}{h}\\
\nonumber&+\frac{(\omega^{11})^2\nabla_i\omega_{11}\nabla_jh}{h^2}+\frac{(\omega^{11})^2\nabla_j\omega_{11}\nabla_ih-\omega^{11}\nabla_{ij}h}{h^2}
\nonumber+\frac{2\omega^{11}\nabla_ih\nabla_jh}{h^3}\bigg)\\
\nonumber=&\frac{(-\omega^{11})^2\nabla_{ij}G\sigma_{n-k}-(\omega^{11})^2\nabla_iG\nabla_j\sigma_{n-k}-(\omega^{11})^2\nabla_jG\nabla_i\sigma_{n-k}}{h}\\
\nonumber&-\frac{(n+1-k)(\omega^{11})^2G\sigma_{n-k}\delta_{ij}-\eta(t)\omega^{11}}{h}\\
\nonumber&-\frac{G(\omega^{11})^2\bigg(d_{ij,ls}\nabla_j(\omega_{ls})\nabla_i(\omega_{11})+d_{ij}\nabla_{ij}\omega_{11}-d_{ij}\delta_{ij}\omega_{11}\bigg)}{h}-\frac{\omega^{11}h_t}{h^2}\\
\nonumber&-Gd_{ij}\bigg(\frac{2(\omega^{11})^3\nabla_i\omega_{11}\nabla_j\omega_{11}-(\omega^{11})^2\nabla_{ij}\omega_{11}}{h}+\frac{(\omega^{11})^2\nabla_i\omega_{11}\nabla_jh}{h^2}\\
\nonumber&+\frac{(\omega^{11})^2\nabla_j\omega_{11}\nabla_ih-\omega^{11}\nabla_{ij}h}{h^2}+\frac{2\omega^{11}\nabla_ih\nabla_jh}{h^3}\bigg)\\
\nonumber=&\frac{-(\omega^{11})^2\nabla_{ij}G\sigma_{n-k}}{h}-\frac{(\omega^{11})^2(\nabla_iG\nabla_j\sigma_{n-k}+\nabla_jG\nabla_i\sigma_{n-k})}{h}\\
\nonumber&-\frac{(n+1-k)(\omega^{11})^2G\sigma_{n-k}\delta_{ij}-\eta(t)\omega^{11}}{h}\\
\nonumber&-\frac{G(\omega^{11})^2\bigg(d_{ij,ls}\nabla_j(\omega_{ls})\nabla_i(\omega_{11})+2d_{ij}\omega^{11}\nabla_i\omega_{11}\nabla_j\omega_{11}\bigg)}{h}\\
\nonumber&+G(\omega^{11})^2\frac{d_{ij}\delta_{ij}\omega_{11}}{h}-\frac{\omega^{11}h_t}{h^2}-Gd_{ij}\bigg(\frac{(\omega^{11})^2\nabla_i\omega_{11}\nabla_jh}{h^2}\\
\nonumber&+\frac{(\omega^{11})^2\nabla_j\omega_{11}\nabla_ih-\omega^{11}\nabla_{ij}h}{h^2}+\frac{2\omega^{11}\nabla_ih\nabla_jh}{h^3}\bigg).
\end{align}

From the reverse concavity of $(\sigma_{n-k})^{\frac{1}{n-k}}$, therefore, by \cite[(3.49)]{UR}, we can derive
\begin{align}\label{eq423}
(d_{ij,mn}+2d_{im}\omega^{nj})\nabla_1\omega_{ij}\nabla_1\omega_{mn}\geq \frac{n+1-k}{n-k}\frac{(\nabla_1\sigma_{n-k})^2}{\sigma_{n-k}}.
\end{align}
Moreover, according to Schwartz inequality, the following result is true,
\begin{align}\label{eq424}
2|\nabla_1\sigma_{n-k}\nabla_1G|\leq\frac{n+1-k}{n-k}\frac{G(\nabla_1\sigma_{n-k})^2}{\sigma_{n-k}}+\frac{n-k}{n+1-k}\frac{\sigma_{n-k}(\nabla_1G)^2}{G}.
\end{align}
Thus, at point $x_2$, substituting (\ref{eq420}), (\ref{eq423}) and (\ref{eq424}) into (\ref{eq422}), we get
\begin{align}\label{eq425}
\nonumber&\frac{\partial(\frac{\omega^{11}}{h})}{\partial t}-Gd_{ij}\nabla_{ij}\frac{\omega^{11}}{h}\\
\nonumber\leq&-\frac{(\omega^{11})^2\nabla_{ij}G\sigma_{n-k}}{h}+\frac{(\omega^{11})^2}{h}\bigg(\frac{n+1-k}{n-k}\frac{G(\nabla_1\sigma_{n-k})^2}{\sigma_{n-k}}
+\frac{n-k}{n+1-k}\frac{\sigma_{n-k}(\nabla_1G)^2}{G}\bigg)\\
\nonumber&-\frac{(n+1-k)(\omega^{11})^2G\sigma_{n-k}\delta_{ij}}{h}+2\frac{\eta(t)\omega^{11}}{h}\\
\nonumber&-\frac{G(\omega^{11})^2}{h}\frac{n+1-k}{n-k}\frac{(\nabla_1\sigma_{n-k})^2}{\sigma_{n-k}}+\frac{(n-k)G\sigma_{n-k}(\omega^{11})^2}{h}\\
\nonumber&-Gd_{ij}\bigg(\frac{(\omega^{11})^2\nabla_i\omega_{11}\nabla_jh}{h^2}+\frac{(\omega^{11})^2\nabla_j\omega_{11}\nabla_ih-\omega^{11}\nabla_{ij}h}{h^2}+\frac{2\omega^{11}\nabla_ih\nabla_jh}{h^3}\bigg)\\
\nonumber\leq&-\frac{(\omega^{11})^2\nabla_{ij}G\sigma_{n-k}}{h}+\frac{(\omega^{11})^2\sigma_{n-k}}{h}\bigg(\frac{n-k}{n+1-k}\frac{(\nabla_1G)^2}{G}\bigg)\\
\nonumber&-\frac{(n+1-k)(\omega^{11})^2G\sigma_{n-k}\delta_{ij}}{h}+2\frac{\eta(t)\omega^{11}}{h}\\
\nonumber&+\frac{(n-k)G\sigma_{n-k}(\omega^{11})^2}{h}+(n-k)G\sigma_{n-k}(\omega^{11})^2\bigg(\frac{\omega_{11}-h}{h^2}\bigg)\\
\leq&-\frac{(\omega^{11})^2}{h}\bigg[\nabla_{11}G\sigma_{n-k}-\frac{n-k}{n+1-k}\sigma_{n-k}\frac{(\nabla_1G)^2}{G}+(n+1-k)G\sigma_{n-k}\bigg]\\
\nonumber&+2\frac{\eta(t)\omega^{11}}{h}-(k-n-1)\frac{\omega^{11}}{h^2}G\sigma_{n-k}.
\end{align}
We use the inequality in \cite{ZR}: $$2\nabla_1G(\sigma_{n-k})_1\leq(n-k)G(n+1-k)G^{n-k-2}(\nabla_1G)^2+\frac{n-k}{n+1-k}\sigma_{n-k}\frac{(\nabla_1G)^2}{G}.$$ By using $\nabla_i\nabla_jG^{\frac{1}{n+1-k}}+G^{\frac{1}{n+1-k}}\delta_{ij}>0$, then we have
\begin{align}\label{eq426}
\frac{1}{n+1-k}\nabla_1\nabla_1G+\frac{1}{n+1-k}\bigg(\frac{1}{n+1-k}-1\bigg)\frac{(\nabla_1G)^2}{G}+G>0.
\end{align}
Inserting (\ref{eq426}) into (\ref{eq425}), by the uniform bounds on $f$, $h$, $\lambda(t)$, $|Du|$ and $\sigma_{n-k}$, we conclude there exists $c_0,c>0$ such that
\begin{align*}
\frac{\partial(\frac{\omega^{11}}{h})}{\partial t}-Gd_{ij}\nabla_{ij}\frac{\omega^{11}}{h}\leq -c_0\frac{(\omega^{11})^2}{h}+c\frac{\omega^{11}}{h}.
\end{align*}

Therefore, $\omega^{11}(x,t)$ has a uniform upper bound, which means that the principal radii is bounded from below by a positive constant $c_1$. In addition, from Lemma \ref{lem47}, we know that for minimal eigenvalue $\lambda_{\min}=\frac{1}{\kappa_{\max}}$ of $\frac{\omega^{11}(x,t)}{h}$ at point $x_2$,
\begin{align*}
C_1\geq\sigma_{n-k}=&\lambda_{\max}\sigma_{n-k-1}(\lambda|\lambda_{\max})+\sigma_{n-k}(\lambda|\lambda_{\max})\\
\geq &\lambda_{\max}\sigma_{n-k-1}(\lambda|\lambda_{\max})\\
\geq &C^{n-k-1}_{n-1}\lambda_{\min}^{n-k-1}\lambda_{\max}\\
\geq &C^{n-k-1}_{n-1}c_1^{n-k-1}\lambda_{\max}
\end{align*}
for constant $C$. Consequently, the principal radii of curvature has uniform upper and lower bounds. This completes the proof of Lemma \ref{lem48}.
\end{proof}
\vskip 0pt
\section{\bf The convergence of the flow}\label{sec5}

With the help of priori estimates in the section \ref{sec4}, the long-time existence and asymptotic behaviour of the flow (\ref{eq110}) are obtained, we also complete the proof of Theorem \ref{thm12}.
\begin{proof}[Proof of Theorem \ref{thm12}] Since the equation (\ref{eq301}) is parabolic, we can get its short time existence. Let $T$ be the maximal time such that $h(\cdot, t)$ is a smooth solution to the equation (\ref{eq301}) for all $t\in[0,T)$. Lemmas \ref{lem42}-\ref{lem47} enable us to apply Lemma \ref{lem48} to the equation (\ref{eq301}), thus, we can deduce an uniformly upper bound and an uniformly lower bound for the biggest eigenvalue of $\{(h_{ij}+h\delta_{ij})(x,t)\}$. This implies
\begin{align*}
	C^{-1}I\leq (h_{ij}+h\delta_{ij})(x,t)\leq CI,\quad \forall (x,t)\in S^{n-1}\times [0,T),
\end{align*}
where $C>0$ is independent of $t$. This shows that the equation (\ref{eq301}) is uniformly parabolic. Estimates for higher derivatives follow from the standard regularity theory of uniformly parabolic equations
Krylov \cite{KR}. Hence, we obtain the long time existence and regularity of solutions for the flow
(\ref{eq110}). Moreover, we obtain
\begin{align}\label{eq501}
\|h\|_{C^{l,m}_{x,t}(S^{n-1}\times [0,T))}\leq C_{l,m},
\end{align}
where $C_{l,m}$ ($l, m$ are nonnegative integers pairs) are independent of $t$, then $T=\infty$. Using the parabolic comparison principle, we can attain the uniqueness of smooth solutions $h(\cdot,t)$ of the equation (\ref{eq301}).

Now, we recall the property of non-decreasing of $\widetilde{T}_k(\Omega_t)$ in Lemma \ref{lem32}, we know that

\begin{align}\label{eq502}
\frac{\partial\widetilde{T}_k(\Omega_t)}{\partial t}\geq 0.
\end{align}
Based on (\ref{eq502}), there exists a $t_0$ such that
\begin{align*}
\frac{\partial\widetilde{T}_k(\Omega_t)}{\partial t}\bigg|_{t=t_0}=0,
\end{align*}
this yields
\begin{align*}
\tau|Du|^{k+1}\sigma_{n-k}=f.
\end{align*}
Let $\Omega=\Omega_{t_0}$, thus, the support function of $\Omega$ satisfies (\ref{eq109}).

In view of (\ref{eq501}), and applying the Arzel$\grave{\textrm{a}}$-Ascoli theorem \cite{BRE} and a diagonal argument, we can extract
a subsequence of $t$, it is denoted by $\{t_j\}_{j \in \mathbb{N}}\subset (0,+\infty)$, and there exists a smooth function $\bar{h}(x)$ such that
\begin{align}\label{eq603}
\|h(x,t_j)-\bar{h}(x)\|_{C^l(S^{n-1})}\rightarrow 0,
\end{align}
uniformly for each nonnegative integer $l$ as $t_j \rightarrow \infty$. This reveals that $\bar{h}(x)$ is a support function. Let us denote by $\Omega$ the convex body determined by $\bar{h}(x)$. Thus, $\Omega$ is a smooth strictly convex body.

Moreover, by (\ref{eq501}) and the uniform estimates in section \ref{sec4}, we conclude that $\widetilde{T}_k(\Omega_t)$ is a bounded function in $t$ and $\frac{\partial \widetilde{T}_k(\Omega_t)}{\partial t}$ is uniformly continuous. Thus, for any $t>0$, by the monotonicity of $\widetilde{T}_k(\Omega_t)$ in Lemma \ref{lem32}, there is a constant $C>0$ being independent of $t$, such that
\begin{align*}
\int_0^t\bigg(\frac{\partial\widetilde{T}_k(\Omega_t)}{\partial t}\bigg)dt=\widetilde{T}_k(\Omega_t)-\widetilde{T}_k(\Omega_0))\leq C,
\end{align*}
this gives
\begin{align}\label{eq504}
\lim_{t\rightarrow\infty}\widetilde{T}_k(\Omega_t)-\widetilde{T}_k(\Omega_0)=
\int_0^\infty\frac{\partial}{\partial t}\widetilde{T}_k(\Omega_t)dt\leq C.
\end{align}
The left hand side of (\ref{eq504}) is bounded, therefore, there is a subsequence $t_j\rightarrow\infty$ such that
\begin{align*}
\frac{\partial}{\partial t}\widetilde{T}_k(\Omega_{t_j})\rightarrow 0 \quad\text{as}\quad  t_j\rightarrow\infty.
\end{align*}
The proof of Lemma \ref{lem32} shows that
\begin{align}\label{eq505}
&\frac{\partial\widetilde{T}_k(\Omega_t)}{\partial t}\bigg|_{t=t_j}\\
\nonumber&=\frac{1}{k}\bigg(\!\int_{S^{n-1}}\frac{h}{f(x)}|Du|^{2(k+1)}(\sigma_{n-k})^2dx
\!-\!\frac{\int_{S^{n-1}}h\sigma_{n-k}|D u|^{k+1}dx}{\int_{S^{n-1}}hf(x)dx}\!\int_{S^{n-1}}h|Du|^{k+1}\sigma_{n-k}dx\bigg)\geq 0.
\end{align}
Taking the limit $t_j\rightarrow\infty$, by equality condition of (\ref{eq505}), it means that there has
\begin{align*}
\tau|Du(X^\infty)|^{k+1}\sigma_{n-k}(h^\infty_{ij}+h^\infty\delta_{ij})=f(x),
\end{align*}
which satisfies (\ref{eq109}). From (\ref{eq501}) and the Arzel$\grave{\textrm{a}}$-Ascoli theorem, we know taht $h^\infty$ is
the support function and the convex body determined by $h^\infty$ is denoted as $\Omega_\infty$. Here, $X^{\infty}=\overline{\nabla} h^{\infty}$ and $\frac{1}{\tau}=\lim_{t_j\rightarrow\infty}\eta(t_j)$. This completes the proof of Theorem \ref{thm12}.
\end{proof}

\end{document}